\pdfoutput=1
\documentclass[11pt]{article}

\usepackage[T1]{fontenc}
\usepackage{lmodern}
\usepackage{amsmath,amssymb,amsthm}
\usepackage{mathtools}
\usepackage{microtype}
\usepackage[margin=3.05cm]{geometry}
\usepackage{booktabs,array}
\usepackage[hypcap=false]{caption}
\usepackage{tikz}
\usepackage[shortlabels]{enumitem}
\usepackage{xcolor}
\usepackage{xurl}
\usepackage[
  unicode=true,
  colorlinks=true,
  linkcolor=blue!50!black,
  citecolor=blue!50!black,
  urlcolor=blue!50!black
]{hyperref}
\hypersetup{
  pdftitle={A finiteness theorem for geodesic Leech wheels},
  pdfauthor={Junyeop Yim},
  pdfsubject={Geodesic Leech labelings of wheel graphs},
  pdfkeywords={geodesic Leech labeling; geodesic path number; non-geodesic Leech graph; wheel graph; Sidon set; additive basis; generating-function method}
}
\allowdisplaybreaks

\newtheorem{theorem}{Theorem}[section]
\newtheorem{proposition}[theorem]{Proposition}
\newtheorem{lemma}[theorem]{Lemma}
\newtheorem{corollary}[theorem]{Corollary}
\theoremstyle{definition}
\newtheorem{definition}[theorem]{Definition}
\theoremstyle{remark}
\newtheorem{remark}[theorem]{Remark}
\numberwithin{equation}{section}

\newcommand{\N}{\mathbb{N}}
\newcommand{\Z}{\mathbb{Z}}
\newcommand{\R}{\mathbb{R}}
\newcommand{\calE}{\mathcal{E}}
\newcommand{\wt}{\operatorname{wt}}
\newcommand{\dd}{\mathbin{\uplus}}
\DeclarePairedDelimiter{\set}{\{}{\}}
\DeclarePairedDelimiter{\abs}{\lvert}{\rvert}
\DeclarePairedDelimiter{\floor}{\lfloor}{\rfloor}
\DeclarePairedDelimiter{\ceil}{\lceil}{\rceil}
\newcommand{\cV}{\mathcal V}
\newcommand{\cX}{\mathcal X}
\newcommand{\cY}{\mathcal Y}
\newcommand{\Db}{\overline D}
\newcommand{\e}{\mathrm{e}}

\title{A finiteness theorem for geodesic Leech wheels}
\author{Junyeop Yim\\
\small Department of Applied Mathematics, Kongju National University,\\
\small Gongju, Republic of Korea\\
\small \texttt{junyeobe0315@smail.kongju.ac.kr}}
\date{}

\begin{document}
\maketitle

\begin{abstract}
Let $f$ be a labeling of the edges of a finite graph $G$ by positive integers, and let the weight of a path be the sum of the labels of its edges. The labeling is a \emph{geodesic Leech labeling} if the weights of the geodesics are exactly $1,2,\ldots,t_{gp}(G)$, each occurring once, where $t_{gp}(G)$ is the geodesic path number of $G$. Let $W_n$ be the wheel on $n$ vertices, a hub joined to an $(n-1)$-cycle.

Our main result is an upper bound: if $n\ge5$ and $W_n$ is geodesic Leech, then $n\le40$. The proof quantifies, via a finite Fourier kernel, the Sidon-type structure of the spoke labels, in which only the cyclically adjacent pairs are allowed as defects, and closes the last three cases with a six-variable Parseval argument. In the other direction, explicit labelings of $W_7,\ldots,W_{13}$, found by a computer search, answer in the negative a problem of Lakshmanan S. and Manattu, who had found labelings of $W_5$ and $W_6$ and expected every $W_n$ with $n\ge7$ to be a non-geodesic Leech graph. Writing $\calE$ for the set of $n\ge5$ for which $W_n$ is geodesic Leech, we obtain $\set{5,6,\ldots,13}\subseteq\calE\subseteq\set{5,6,\ldots,40}$.
\end{abstract}

\section{Introduction}\label{sec:intro}

Leech proposed the problem of labeling the edges of a tree so that the weights of all paths form an initial segment of the positive integers~\cite{Leech1975}. Only five Leech trees are known~\cite{OzenWangYalman2016}, and Taylor showed that a Leech tree on $N$ vertices can exist only if $N=k^2$ or $N=k^2+2$~\cite{Taylor1977}. Further nonexistence results are due to Sz\'ekely, Wang, and Zhang~\cite{SzekelyWangZhang2005}, who also conjectured that there are only finitely many Leech trees, to Varghese, Lakshmanan S., and Arumugam~\cite{VargheseEtAl2020}, and, more recently, to Luo and Yu~\cite{LuoYu2024}; see also Leach~\cite{Leach2014} for a modular variant. Varghese, Lakshmanan S., and Arumugam later extended the problem to general graphs and introduced geodesic Leech labelings, which consider only shortest paths instead of all paths~\cite{VargheseEtAl2022,VargheseLeechGraphs2024,VargheseEtAl2024}; see also Lakshmanan S. and Eldho~\cite{LakshmananEldho2024}. Recently, Lakshmanan S. and Manattu investigated several graph families, exhibited geodesic Leech labelings of the wheel graphs $W_5$ and $W_6$, and stated as an open problem their belief that every wheel $W_n$ with $n\ge7$ is a non-geodesic Leech graph~\cite{LakshmananManattu2025}. They noted the first open case: for $W_7$ they gave an \emph{almost} geodesic Leech labeling~\cite[Figure~4]{LakshmananManattu2025} and left open whether $W_7$ is geodesic Leech.

The main result of this paper is an explicit upper bound on the orders $n$ for
which $W_n$ can be geodesic Leech, a restriction on the order in the spirit of
Taylor's for Leech trees.
Throughout, $W_n$ denotes the wheel on $n$ vertices, a hub joined to a cycle of
length $n-1$; this is the convention of~\cite{LakshmananManattu2025}. It differs
from the convention $W_n=C_n+K_1$, on $n+1$ vertices, used in Gallian's
survey~\cite{Gallian}, in~\cite{LakshmananEldho2024}, and in much of the
labeling literature: the wheel called
$W_{40}$ here has $40$ vertices and a rim cycle of length $39$.

\begin{theorem}[Main Theorem]\label{thm:main}
If $n\ge5$ and $W_n$ is geodesic Leech, then $n\le40$.
\end{theorem}

\noindent
Theorem~\ref{thm:main} is proved in Section~\ref{sec:boundary}, at the end of the
argument begun in Section~\ref{sec:strategy}.

The starting point of the proof is the observation that almost all two-element sums of the spoke labels must be pairwise distinct and must remain within a permitted interval. This is a Sidon-type structure in which only the cyclically adjacent pairs are allowed as defects. We develop a finitary version of the generating-function method of Moser~\cite{Moser1960} and of Moser, Pounder, and Riddell~\cite[Lemma~1]{MoserPounderRiddell1969}, in the form used by Pikhurko~\cite[\S4]{Pikhurko2006}, and combine it with the Parseval identity for the residual rim weights and the phase information of the adjacent sums. The essential steps are to handle the infinite range with three uniform inequalities and to close the remaining three boundary cases with a single six-variable Parseval argument. The final numerical comparisons are made between explicit rational numbers and square roots, together with the alternating Taylor inequality for the sine.

In the other direction, $W_n$ is geodesic Leech for $7\le n\le13$, contrary to the expectation expressed in~\cite{LakshmananManattu2025}: Section~\ref{sec:certificates} lists explicit labelings of $W_7,\ldots,W_{13}$, found by a computer search, whose verification is a finite computation (Theorem~\ref{thm:explicit}). Combining these with the main theorem, we obtain the following.

\begin{corollary}\label{cor:range}
Let
\[
  \calE=\set{n\ge5:W_n\text{ is geodesic Leech}}.
\]
Then
\[
  \set{5,6,\ldots,13}\subseteq\calE
  \subseteq\set{5,6,\ldots,40}.
\]
\end{corollary}

\noindent
In particular $\calE$ is finite. This is the analogue for wheels of the
finiteness conjectured by Sz\'ekely, Wang, and Zhang~\cite{SzekelyWangZhang2005}
for Leech trees; whereas Taylor's restriction~\cite{Taylor1977} is a congruence
condition on the order, Corollary~\ref{cor:range} is an upper bound.

\begin{remark}\label{rem:not-sharp}
The value $40$ is the threshold of a single comparison, the one made in the case
of no large spoke (Section~\ref{ssec:h0}); the other estimates of the proof do
not enter it, and their constants have not been optimized. We do not claim that
$40$ is best possible. What the method does and does not give, and why lowering
the bound by the present method requires more than sharpening any one of the
estimates used here, is discussed in Section~\ref{sec:conclusion}.
\end{remark}

No monotonicity is known by which nonexistence at one order would imply nonexistence at all larger orders. The final classification problem is therefore the finite but nonmonotone problem of deciding each of $W_{14},\ldots,W_{40}$.

The paper is organized as follows. Section~\ref{sec:wheels} classifies the geodesics of wheels, and Section~\ref{sec:frame} proves a criterion that decomposes the problem into the choice of an admissible spoke frame and a residual cyclic completion. Sections~\ref{sec:strategy}--\ref{sec:boundary} form the proof of the upper bound: the proof strategy and the three cases according to the number of large spokes (Section~\ref{sec:strategy}); the finite Fourier kernel and the cyclic-defect interval-occupancy bound (Section~\ref{sec:fourier}); the support polynomial and the localization of the large spokes (Section~\ref{sec:support-localization}); the uniform inequalities eliminating the infinite range (Section~\ref{sec:uniform-fourier}); and the six-variable Parseval argument closing the remaining three boundary cases (Section~\ref{sec:boundary}). Section~\ref{sec:certificates} presents the explicit constructions for $W_7,\ldots,W_{13}$ and completes the proof of Corollary~\ref{cor:range}. Appendix~\ref{app:partitions} contains the three geodesic weight classes of each explicit construction, Appendix~\ref{app:identities} records the arithmetic constraints that underlie the search of Section~\ref{sec:certificates} and reduce the search space for the remaining cases $W_{14},\ldots,W_{40}$, Appendix~\ref{app:kappa-iteration} exhibits the fixed-point iteration behind Table~\ref{tab:kappa-maxima}, and Appendix~\ref{app:tangent} records the data behind the tangent planes of Section~\ref{sec:boundary}.

\section{Geodesics of wheel graphs}\label{sec:wheels}

Throughout, all graphs are finite, simple, undirected, and connected. For a positive integer $r$ we write $[r]=\{1,\ldots,r\}$. A shortest path between two distinct vertices is called a \emph{geodesic}. A path and its reverse are regarded as the same, but shortest paths with the same endpoints and different interior vertices or edges are counted separately. Let $t_{gp}(G)$ denote the number of geodesics of $G$; this quantity is the \emph{geodesic path number} of $G$~\cite{VargheseEtAl2022,LakshmananManattu2025}. The same parameter has recently been studied by Knor, Sedlar, \v{S}krekovski, and Zhang~\cite{KnorEtAl2026}, under the name geodesic subpath number and the notation $\mathrm{gpn}(G)$.

For an edge labeling $f:E(G)\to\N$ and a path $P$, we call
\[
  \wt_f(P)=\sum_{e\in E(P)}f(e)
\]
the weight of $P$. If the multiset of all geodesic weights is exactly $[t_{gp}(G)]$, then $f$ is called a \emph{geodesic Leech labeling}, and a graph admitting one is called a \emph{geodesic Leech graph}~\cite{VargheseEtAl2022}; we also say simply that $G$ \emph{is geodesic Leech}. A graph admitting no such labeling is called a \emph{non-geodesic Leech graph}~\cite{LakshmananManattu2025}.

Following~\cite{LakshmananManattu2025}, we define $W_n$ to be the wheel graph on $n$ vertices. Thus $W_{m+1}$ has a hub $u$ and rim vertices
\[
  v_0,v_1,\ldots,v_{m-1}
\]
in cyclic order, with all indices read in $\Z/m\Z$; we write $C_m$ for the rim cycle. We assume throughout that $m\ge4$, that is, $n\ge5$. In the rest of the paper we write
\[
  m=n-1
\]
for the length of the rim, so that $W_n=W_{m+1}$; Sections~\ref{sec:wheels}--\ref{sec:certificates} are carried out in the variable $m$, while the statements of Sections~\ref{sec:intro} and~\ref{sec:conclusion} are in $n$. For a labeling $f$ we set
\[
  a_i=f(uv_i),\qquad b_i=f(v_iv_{i+1}),
\]
and call these the spoke labels and the rim labels, respectively.

\begin{proposition}\label{prop:wheel-geodesics}
Every geodesic of $W_{m+1}$ belongs to exactly one of the following four classes.
\begin{align*}
&uv_i, &&\text{weight }a_i;\\
&v_iv_{i+1}, &&\text{weight }b_i;\\
&v_iv_{i+1}v_{i+2}, &&\text{weight }b_i+b_{i+1};\\
&v_i u v_j\quad(0\le i<j<m,\ \{i,j\}\notin E(C_m)),
&&\text{weight }a_i+a_j.
\end{align*}
Consequently
\[
  t_{gp}(W_{m+1})=\frac{m(m+3)}2.
\]
Equivalently, for $n\ge5$,
\[
  t_{gp}(W_n)=\frac{(n-1)(n+2)}2.
\]
\end{proposition}

\begin{proof}
The count $t_{gp}(W_n)=(n-1)(n+2)/2$ is due to Lakshmanan S. and
Manattu~\cite[Theorem~3.1]{LakshmananManattu2025}; we reproduce the argument
because the classification of the geodesics into the four classes above, and not
only their number, is used throughout what follows.
The wheel has diameter 2. The geodesics of length 1 are the $m$ spokes and the $m$ rim edges. A path of length 2 formed by two consecutive rim edges is a geodesic, since its two endpoints are nonadjacent; this is where the standing hypothesis $m\ge4$ is used, for it guarantees that $v_i$ and $v_{i+2}$ are distinct and nonadjacent, whereas at $m=3$ the wheel is $K_4$ and has diameter 1. A path of length 2 through the hub is a geodesic only when its two rim endpoints are nonadjacent. A mixed path traversing one spoke and one rim edge in succession is not a geodesic, since its endpoints are already adjacent via a spoke. This exhausts all cases.

The first three classes contribute $m$ geodesics each, and the last contributes
\[
  \binom m2-m=\frac{m(m-3)}2.
\]
Summing gives the stated formula.
\end{proof}

We call the geodesics of the fourth class, the two-edge geodesics through the hub, \emph{hub-type} geodesics. Now set
\[
  N=\frac{m(m+3)}2.
\]
By Proposition~\ref{prop:wheel-geodesics}, the geodesic Leech condition is equivalent to
\begin{equation}\label{eq:wheel-master}
 \set{a_i}
 \dd \set{b_i}
 \dd \set{b_i+b_{i+1}}
 \dd \set{a_i+a_j:0\le i<j<m,\ \{i,j\}\notin E(C_m)}
 = [N],
\end{equation}
where each brace denotes the multiset of the indexed values and $\dd$ is multiset union; since the right-hand side is a set of $N$ distinct elements, \eqref{eq:wheel-master} asserts in particular that all the listed values are pairwise distinct.

\section{Spoke frames and residual cyclic completion}\label{sec:frame}

For a cyclic sequence $A=(a_0,\ldots,a_{m-1})$ define
\[
  \Phi(A)=
  \set{a_i:0\le i<m}
  \cup
  \set{a_i+a_j:0\le i<j<m,\ \{i,j\}\notin E(C_m)}.
\]

\begin{definition}
We call $A$ an \emph{admissible spoke frame}, or simply an \emph{admissible frame}, if all the indexed values
\[
  a_i,
  \qquad
  a_i+a_j\quad(0\le i<j<m,\ \{i,j\}\notin E(C_m))
\]
lie in $[N]$ and are pairwise distinct.
\end{definition}

An admissible frame satisfies
\[
  \abs{\Phi(A)}=m+\frac{m(m-3)}2=\frac{m(m-1)}2.
\]
Hence the residual set
\[
  T(A)=[N]\setminus\Phi(A)
\]
has size exactly $2m$.

For an admissible spoke frame $A$, a cyclic sequence $B=(b_0,\ldots,b_{m-1})$ of positive integers is called a \emph{residual cyclic completion} of $A$ if
\begin{equation}\label{eq:completion}
  T(A)=\set{b_i:i\in\Z/m\Z}
  \dd
  \set{b_i+b_{i+1}:i\in\Z/m\Z},
\end{equation}
where, as in \eqref{eq:wheel-master}, $\dd$ denotes multiset union. Since $T(A)$ is a set of $2m$ distinct integers, \eqref{eq:completion} asserts in particular that the $m$ values $b_i$ and the $m$ values $b_i+b_{i+1}$ are pairwise distinct; the distinctness of the $b_i$ is therefore automatic and is not imposed as a hypothesis.

\begin{proposition}[Frame-completion criterion]\label{thm:frame-completion}
The wheel $W_{m+1}$ is geodesic Leech if and only if some admissible spoke frame admits a residual cyclic completion.
\end{proposition}

\begin{proof}
Let $A$ be an admissible spoke frame and $B$ a residual cyclic completion of $A$. The set $\Phi(A)$ consists exactly of the weights of the spokes themselves and of the length-2 geodesics through the hub, and \eqref{eq:completion} says that the remaining $2m$ numbers are partitioned into the rim edge labels and the sums of consecutive pairs of rim edges. These four classes agree exactly with the four classes of Proposition~\ref{prop:wheel-geodesics}, so the labeling with spoke labels $A$ and rim labels $B$ is geodesic Leech. Conversely, extracting the spoke labels from any geodesic Leech wheel yields an admissible frame, and the rim labels form a residual cyclic completion of it.
\end{proof}

Proposition~\ref{thm:frame-completion} decomposes the problem into two successive steps: the choice of an admissible spoke frame, and the choice of a residual cyclic completion of it. Every geodesic Leech labeling of $W_{m+1}$ arises in this way. This decomposition organizes both the proof of the upper bound, which begins in the next section, and the search behind the constructions of Section~\ref{sec:certificates}.

\section{Strategy of the upper-bound proof and large spokes}\label{sec:strategy}

The strategy of the proof of Theorem~\ref{thm:main} is summarized as follows. By \eqref{eq:wheel-master}, the sums of spoke pairs nonadjacent on the rim must all be distinct, so the set of spoke labels lies in the interval $[1,N]$ as a Sidon-type set in which only the $m$ cyclically adjacent pairs are allowed as defects. Recall that a \emph{Sidon set} is a set of integers whose pairwise sums are all distinct. A Sidon subset of $[N]$ has at most $\sqrt N+O(N^{1/4})$ elements by Erd\H os and Tur\'an~\cite{ErdosTuran1941}, and at most $\sqrt N+N^{1/4}+1$ elements for every $N$ by Lindstr\"om~\cite{Lindstrom1969}; see Balogh, F\"uredi, and Roy~\cite{BaloghFurediRoy2023} for a recent improvement and O'Bryant~\cite{OBryant2004} for a survey. What is needed below is a quantitative bound of this kind for sets with $m$ prescribed defects, valid at each finite $N$. The finite Fourier kernel of Section~\ref{sec:fourier} produces a gap for such sets, giving a lower bound linear in the number of elements and an upper bound of order $\sqrt m$ for the Fourier coefficients; the cyclic-defect interval-occupancy bound of the same section, together with the geodesic support polynomial of Section~\ref{sec:support-localization}, localizes the large spokes, showing that each exceeds $N/2$ by at most $3m/2$, and thereby extends this gap to all cases. The uniform inequalities of Section~\ref{sec:uniform-fourier} eliminate every $m\ge40$ except three boundary cases, and the six-variable Parseval argument of Section~\ref{sec:boundary} closes the remaining three.

Every numerical comparison in this paper reduces to inequalities between
explicit rational numbers and square roots, verified by squaring both sides,
together with rational Taylor bounds for the sine and rational brackets for
$\pi$. Two notations for decimals are used, with different meanings. A decimal
followed by $\ldots$ denotes the exact value truncated at the last digit shown,
so that $u_{40}=16.8226\ldots$ asserts $16.8226\le u_{40}<16.8227$. A decimal
carrying an inequality sign is an exact rational bound, so that
$\mathcal U_0(40)<39.943$ asserts exactly that. Where a displayed value is
instead rounded, as in Tables~\ref{tab:app-values}
and~\ref{tab:app-gradients}, this is stated at that point.

The first step is an elementary observation limiting the number of large spokes.

\begin{lemma}[Large spokes]\label{lem:large-spokes}
At most two spoke labels exceed $\floor{N/2}$; if two do, then the corresponding spokes have adjacent rim endpoints.
\end{lemma}

\begin{proof}
Among any three vertices of a cycle with $m\ge4$ there are two nonadjacent ones. If three spokes exceeded $\floor{N/2}$, then two nonadjacent ones among them would have sum at least
\[
  2\bigl(\floor{N/2}+1\bigr)>N,
\]
exceeding the permitted range of hub-type geodesics; the inequality holds for both parities of $N$, giving $N+2>N$ when $N$ is even and $N+1>N$ when $N$ is odd. The same argument shows that two spokes exceeding $\floor{N/2}$ cannot be nonadjacent.
\end{proof}

A spoke label exceeding $\floor{N/2}$ will be called a \emph{large spoke}.
Since the labels are integers, $a_i>\floor{N/2}$ is the same condition as
$a_i>N/2$ for both parities of $N$, and we use the two interchangeably. From now
on we write $h$ for the number of large spokes; by
Lemma~\ref{lem:large-spokes}, $h\in\{0,1,2\}$. The proof splits into these three
cases, and the way they are distributed over Sections~\ref{sec:uniform-fourier}
and~\ref{sec:boundary} is recorded in Table~\ref{tab:case-structure}.

\begin{table}[ht]
\centering
\caption{The structure of the proof of Theorem~\ref{thm:main}. Each row is one
value of $h$; the third column names the comparison that eliminates the
infinite range, and the last column the finitely many cases it
leaves.}\label{tab:case-structure}
\begin{tabular}{@{}clll@{}}
\toprule
$h$ & uniform range excluded & governing comparison & boundary cases left\\
\midrule
$0$ & $m\ge40$ (\S\ref{ssec:h0})
    & $\mathcal L_0>\mathcal U_0$
    & ---\\
$1$ & $m\ge41$ (\S\ref{ssec:h1})
    & $\mathcal L_1>\mathcal U_0$
    & $(40,1)$\\
$2$ & $m\ge42$ (\S\ref{ssec:h2})
    & $\mathcal L_2>\mathcal U_2$
    & $(40,2)$, $(41,2)$\\
\bottomrule
\end{tabular}
\end{table}

The three boundary cases $(m,h)$ in the last column are closed in
Section~\ref{sec:boundary}, and together with
Proposition~\ref{prop:uniform-ranges} this leaves $m\le39$, that is, $n\le40$.
The lower bounds $\mathcal L_1$ and $\mathcal L_2$, and the upper bound
$\mathcal U_2$, are the ones that require the localization of the large spokes
in Section~\ref{sec:support-localization}; the case $h=0$ uses only the kernel
of Section~\ref{sec:fourier} and the phase-pair inequality of
Section~\ref{sec:uniform-fourier}.

\section{The finite Fourier kernel and cyclic-defect interval occupancy}\label{sec:fourier}

In this section we prepare a finite-length version of the kernel underlying the generating-function method of Moser~\cite{Moser1960} and of Moser, Pounder, and Riddell~\cite[Lemma~1]{MoserPounderRiddell1969}, as used by Pikhurko~\cite[\S4]{Pikhurko2006}, and apply it to the situation where only sums of distinct element pairs are counted, deriving a cyclic-defect interval-occupancy upper bound. This bound expresses quantitatively the fact that the only defects permitted among the spoke sums are the adjacent pairs of a single cycle.

\subsection{The finite Fourier kernel}

For $r\ge1$ let
\[
  c_r=\frac{2}{4r^2-1}
  =\frac1{2r-1}-\frac1{2r+1}.
\]
From the Fourier series of $\abs{\sin x}$, which appears in this form in~\cite[p.~400]{MoserPounderRiddell1969} and is used in this way by Pikhurko~\cite[\S4]{Pikhurko2006}, we obtain
\begin{equation}\label{eq:fourier-identity}
  \frac\pi2\sin x+
  \sum_{r=1}^{\infty}c_r\cos(2rx)
  =
  \begin{cases}
    1,&0\le x\le\pi,\\
    1+\pi\sin x,&\pi\le x\le2\pi,
  \end{cases}
\end{equation}
where the right-hand side is extended $2\pi$-periodically. For an integer $R\ge1$ define the finite kernel
\[
  K_R(x)=\frac\pi2\sin x+\sum_{r=1}^{R}c_r\cos(2rx).
\]
Telescoping gives
\begin{equation}\label{eq:coeff-sums}
  \rho_R:=\sum_{r=1}^{R}c_r=1-\frac1{2R+1},
  \qquad
  \sum_{r=R+1}^{\infty}c_r=\frac1{2R+1}.
\end{equation}

\begin{lemma}[Kernel lower bound]\label{lem:kernel}
For all $R\ge1$,
\[
  K_R(x)\ge\rho_R\quad(0\le x\le\pi),
  \qquad
  K_R(x)\ge\rho_R+\pi\sin x\quad(\pi\le x\le2\pi).
\]
\end{lemma}

\begin{proof}
By \eqref{eq:coeff-sums}, the tail $\sum_{r>R}c_r\cos(2rx)$ has absolute value at most $1-\rho_R$. Subtracting this tail from the values $1$ and $1+\pi\sin x$ of the infinite series on the two intervals of \eqref{eq:fourier-identity} yields the stated lower bounds.
\end{proof}

When the kernel is summed over the phases of a set, the resulting sum is bounded below linearly in the number of elements and above by a common bound on the Fourier coefficients used. The following is the form used in this paper.

\begin{lemma}[Kernel sum inequality]\label{lem:kernel-sum}
Let $M\ge3$, $R=\floor{(M-1)/2}$, and $\rho=\rho_R$, so that $R\ge1$ and the kernel $K_R$ of Lemma~\ref{lem:kernel} is defined. Suppose $S\subseteq\{0,1,\ldots,M-1\}$ has $\abs S$ elements, of which $h$ exceed $R$. With $\zeta=\e^{2\pi i/M}$ and $F(z)=\sum_{a\in S}z^a$, if some $U>0$ satisfies $\abs{F(\zeta^t)}\le U$ for all $t\in\{1\}\cup\{2,4,6,\ldots,2R\}$, then
\begin{equation}\label{eq:finite-sum-lower}
  \left(\frac\pi2+\rho\right)U
  \ge
  \sum_{a\in S}K_R\!\left(\frac{2\pi a}{M}\right)
  \ge\rho\abs S-\pi h.
\end{equation}
\end{lemma}

\begin{proof}
For each $a\in S$ set $x_a=2\pi a/M$. The frequencies $1,2,4,\ldots,2R$ are all nonzero modulo $M$. Since $\sum_a\sin x_a=\operatorname{Im}F(\zeta)$ and $\sum_a\cos(2rx_a)=\operatorname{Re}F(\zeta^{2r})$, the left inequality follows from $c_r>0$ and the hypothesis. The right inequality follows from Lemma~\ref{lem:kernel}: if $a\le R$ then $0\le x_a\le\pi$, so $K_R(x_a)\ge\rho$; and if $a>R$ then $\pi\le x_a<2\pi$, so $K_R(x_a)\ge\rho+\pi\sin x_a\ge\rho-\pi$.
\end{proof}

\subsection{A finite Fourier inequality for distinct pair sums}

Fix an integer $w\ge2$ and a set $S\subseteq\{0,1,\ldots,w-1\}$, and let $\ell=|S|$. Let
\[
  s^{\times}(S)=
  \abs*{\set{x+y:x,y\in S,\ x<y}}
\]
denote the number of sums of distinct element pairs. Also write
\[
  M_w=2w-1,\qquad
  L_w(\ell)=\frac{2\rho_{w-1}\ell}{\pi+2\rho_{w-1}},
\]
where $\rho_{w-1}=1-1/M_w$ is the coefficient sum of \eqref{eq:coeff-sums}.

\begin{lemma}[Distinct-pair-sum upper bound]\label{lem:distinct-pair}
If $L_w(\ell)\ge\tfrac12$, then
\begin{equation}\label{eq:distinct-pair}
  s^{\times}(S)
  \le
  \frac{M_w+\binom{\ell}{2}}{2}
  -\frac{L_w(\ell)^2-L_w(\ell)}4.
\end{equation}
\end{lemma}

\begin{proof}
Set
\[
  F(z)=\sum_{b\in S}z^b,
  \qquad
  G_{\times}(z)=\frac{F(z)^2-F(z^2)}2.
\]
The coefficients of $G_{\times}$ are the representation numbers by unordered pairs of distinct elements, and all its exponents lie in $\{0,1,\ldots,M_w-1\}$. Let $\xi=\e^{2\pi i/M_w}$. The polynomial $1+z+\cdots+z^{M_w-1}$ vanishes at the nontrivial $M_w$-th roots of unity, so $G_{\times}(\xi^t)$ is the value at $\xi^t$ of $G_{\times}(z)-(1+z+\cdots+z^{M_w-1})$, and the sum of the absolute values of the coefficients of this difference gives
\begin{equation}\label{eq:distinct-Z}
  \abs{G_{\times}(\xi^t)}
  \le Z_{\times}:=\binom{\ell}{2}+M_w-2s^{\times}(S)
  \qquad(t\ne0).
\end{equation}
Moreover, since
\[
  F(z)^2=2G_{\times}(z)+F(z^2),
\]
we obtain
\[
  \abs{F(\xi^t)}^2
  \le2Z_{\times}+\abs{F(\xi^{2t})}.
\]
Since $M_w$ is odd, the doubling map permutes the nontrivial frequencies. Hence, setting
\[
  U=\max_{t\ne0}\abs{F(\xi^t)},
\]
we obtain
\begin{equation}\label{eq:distinct-U}
  U^2\le2Z_{\times}+U.
\end{equation}

Apply Lemma~\ref{lem:kernel-sum} to the set $S$ with $M=M_w=2w-1\ge3$, so that $R=w-1$ and the coefficient sum is $\rho_{w-1}$; since every $b\in S$ is at most $w-1$, there are no elements in the upper half-interval, and $U$ bounds $\abs F$ at all nontrivial frequencies. The hypothesis $U>0$ of Lemma~\ref{lem:kernel-sum} holds, for $U=0$ would force the nonzero polynomial $F$, of degree at most $w-1<M_w-1$, to vanish at all $M_w-1$ nontrivial $M_w$-th roots of unity. Therefore \eqref{eq:finite-sum-lower} gives
\[
  \rho_{w-1}\ell\le\left(\frac\pi2+\rho_{w-1}\right)U,
\]
that is, $U\ge L_w(\ell)$. The function $u\mapsto u^2-u$ is increasing for $u\ge\tfrac12$, so \eqref{eq:distinct-U} yields
\[
  2Z_{\times}\ge L_w(\ell)^2-L_w(\ell).
\]
Substituting this into \eqref{eq:distinct-Z} gives \eqref{eq:distinct-pair}.
\end{proof}

\subsection{Cyclic-defect occupancy numbers}

Throughout, a pair of labeled vertices that is an edge of the underlying path or cycle is called a \emph{defect pair}, and any other pair is called a \emph{non-defect pair}. The terminology records the fact established in Section~\ref{sec:strategy}: the spoke labels of a geodesic Leech wheel would form a Sidon set were it not for the $m$ pairs that are edges of the rim cycle $C_m$, whose sums are exempt from the distinctness requirement because they are not weights of hub-type geodesics. Only the defect pairs are allowed to repeat a sum. Sets of this kind are quasi-Sidon in the sense of Erd\H os and Freud~\cite{ErdosFreud1991}, and the condition that all sums over non-defect pairs be distinct is that of an edge-sum distinguishing labeling, in the sense of Tuza~\cite{Tuza2017} and Bok and Jedli\v{c}kov\'a~\cite{BokJedlickova2021}, of the complement of the path or cycle, with labels drawn from an interval.

Suppose distinct integers $x_1,\ldots,x_s$ are attached to the vertices of a path on $s$ vertices, and suppose the sums over all non-defect pairs are pairwise distinct. Let $\mu_s(w)$ denote the maximum number of labels that can lie in an interval of $w$ consecutive integers, and set
\[
  \mu(w)=\sup_{s\ge1}\mu_s(w),
\]
which is finite because $\mu_s(w)\le w$ for every $s$, the labels being distinct integers of the interval. All the bounds below are bounds on $\mu(w)$, hence hold uniformly in $s$.

Suppose $\ell$ labels lie in such an interval. Among the $\ell$ selected vertices there are at most $\ell-1$ path edges, so at least
\[
  \binom{\ell}{2}-(\ell-1)=\frac{(\ell-1)(\ell-2)}2
\]
pairs are non-defect pairs, and their sums must all be distinct. By an elementary count, at most $\max(2w-3,0)$ distinct sums are possible, and Lemma~\ref{lem:distinct-pair} applies as well.

Two consequences of this count hold for every $w$, with no
appeal to the Fourier input, and we record them for use at small arguments.
First, the $\ell$ labels are distinct integers of the interval, so $\ell\le w$;
second, the $(\ell-1)(\ell-2)/2$ non-defect sums are distinct elements of an interval
of $2w-3$ integers, a constraint that is vacuous for $\ell\le2$, since the guaranteed
count $(\ell-1)(\ell-2)/2$ vanishes there. Hence
\begin{equation}\label{eq:rho-elementary}
  \mu(w)\le\varepsilon(w):=
  \max\set*{\ell\le w:\ \tfrac{(\ell-1)(\ell-2)}2\le\max(2w-3,0)},
\end{equation}
and in particular
\begin{equation}\label{eq:rho-small}
  \bigl(\varepsilon(1),\ldots,\varepsilon(9)\bigr)=(1,2,3,4,5,5,6,6,7).
\end{equation}
Finally, an interval of $w$ consecutive integers is contained in one of $w'\ge w$
consecutive integers, so
\begin{equation}\label{eq:rho-monotone}
  \mu(w)\le\mu(w')\qquad(w\le w').
\end{equation}

\begin{proposition}[Interval-occupancy bounds]\label{lem:rho-closed}
Let $w\ge10$ and set
\[
  \lambda_w=\frac{14(w-1)}{36w-25}.
\]
Then $\ell=\mu(w)$ satisfies
\begin{equation}\label{eq:rho-sharp}
  (1+\lambda_w^2)\ell^2-(5+\lambda_w)\ell+6-4w\le0.
\end{equation}
Relaxing $\lambda_w\ge\tfrac38$ in \eqref{eq:rho-sharp} gives
\begin{equation}\label{eq:rho-quadratic}
  73\ell^2-344\ell+384-256w\le0,
\end{equation}
and hence the closed form
\begin{equation}\label{eq:rho-closed}
  \mu(w)<\frac{15}{8}\sqrt w+\frac{12}{5}.
\end{equation}
Moreover \eqref{eq:rho-closed} holds for every $w\ge1$.
\end{proposition}

\begin{proof}
Let $\ell\le\mu_s(w)$; we may assume $\ell\ge2$, for if $\ell\le1$ then, since
$0<\lambda_w<1$, the left-hand side of \eqref{eq:rho-sharp} is at most
$6-4w<0$ and there is nothing to prove. Combine
Lemma~\ref{lem:distinct-pair}, applied to the set $S$ of the labels lying in the
interval (translated to start at $0$), with
\[
  \frac{(\ell-1)(\ell-2)}2\le s^{\times}(S).
\]
Clearing the denominators of \eqref{eq:distinct-pair}, this reads
\[
  2(\ell-1)(\ell-2)\le2M_w+\ell(\ell-1)-\bigl(L_w(\ell)^2-L_w(\ell)\bigr),
\]
that is, $\ell^2-5\ell+6+L_w(\ell)^2-L_w(\ell)\le4w$. Using $\pi<22/7$,
\[
  \frac{2\rho_{w-1}}{\pi+2\rho_{w-1}}
  \ge \frac{14(w-1)}{36w-25}=\lambda_w,
\]
so $L_w(\ell)\ge\lambda_w\ell\ge\tfrac12$; since $u\mapsto u^2-u$ is increasing for
$u\ge\tfrac12$, substituting $L_w(\ell)\ge\lambda_w\ell$ gives \eqref{eq:rho-sharp}.

For $w\ge10$ we have $\lambda_w\ge\tfrac38$. The two terms of
\eqref{eq:rho-sharp} that involve $\lambda_w$ group as
$(\lambda_w\ell)^2-\lambda_w\ell$, so that the substitution of $\tfrac38$ for
$\lambda_w$ is not termwise, the term $-(5+\lambda_w)\ell$ increasing when
$\lambda_w$ is lowered; but $\ell\ge2$ gives $\tfrac38\ell\ge\tfrac34>\tfrac12$,
and $u\mapsto u^2-u$ is increasing for $u\ge\tfrac12$, so
$(\lambda_w\ell)^2-\lambda_w\ell\ge\bigl(\tfrac38\ell\bigr)^2-\tfrac38\ell$. Hence
$\ell^2-5\ell+6+L_w(\ell)^2-L_w(\ell)\le4w$ also gives
$\bigl(1+\tfrac9{64}\bigr)\ell^2-\bigl(5+\tfrac38\bigr)\ell+6-4w\le0$, which upon
clearing denominators is \eqref{eq:rho-quadratic}. Finally, substituting
\[
  \ell_0=\frac{15}{8}\sqrt w+\frac{12}{5}
\]
into the left-hand side of \eqref{eq:rho-quadratic} yields
\[
  73\ell_0^2-344\ell_0+384-256w
  =\frac{41}{64}w+12\sqrt w-\frac{528}{25},
\]
which is increasing in $w$, negative at $w=1$ and $w=2$, and positive for every
$w\ge3$, hence for $w\ge10$. Since $\ell_0\ge\tfrac{15}8+\tfrac{12}5>\tfrac{172}{73}$,
the point $\ell_0$ lies to the right of the vertex of the parabola in
\eqref{eq:rho-quadratic}, so $\ell<\ell_0$, which is \eqref{eq:rho-closed} for
$w\ge10$. For $1\le w\le9$ the same
bound follows from \eqref{eq:rho-small}, since
$\bigl(\varepsilon(w)-\tfrac{12}5\bigr)^2<\tfrac{225}{64}w$ implies
$\varepsilon(w)-\tfrac{12}5<\tfrac{15}8\sqrt w$ whatever the sign of
$\varepsilon(w)-\tfrac{12}5$, a quantity that is negative exactly at $w=1$ and
$w=2$; the nine pairs
\[
  \Bigl(\varepsilon(w)-\tfrac{12}5\Bigr)^2
  \quad\text{versus}\quad
  \tfrac{225}{64}w
\]
compare as $1.96<3.51$, $0.16<7.03$, $0.36<10.54$, $2.56<14.06$, $6.76<17.57$,
$6.76<21.09$, $12.96<24.60$, $12.96<28.12$, $21.16<31.64$.
\end{proof}

The relaxation of \eqref{eq:rho-sharp} to \eqref{eq:rho-quadratic} is of no consequence in
the asymptotic range, but it is too weak near $w=110$, where it is exactly
the difference between $\mu(w)\le21$ and $\mu(w)\le22$. The boundary
values $m=40,41,42$ in Lemma~\ref{lem:two-large-position} and $m=40,41$ in
Lemma~\ref{lem:one-large-position} therefore use \eqref{eq:rho-sharp} directly;
only the ranges $m\ge43$ in the former and $m\ge42$ in the latter use
\eqref{eq:rho-closed}.

\begin{remark}\label{rem:pikhurko}
Proposition~\ref{lem:rho-closed} is a finite and fully explicit interval-occupancy bound
for sets that are Sidon apart from the edges of a path. It may be compared with the
asymptotic bound
\[
  \abs A\le\Bigl(\bigl(\tfrac14+(\pi+2)^{-2}\bigr)^{-1/2}+o(1)\Bigr)W^{1/2}
  =\bigl(1.863\ldots+o(1)\bigr)W^{1/2}
\]
of Pikhurko~\cite[Theorem~3]{Pikhurko2006} for quasi-Sidon subsets of an interval
of length $W$, in the sense of Erd\H os and Freud~\cite{ErdosFreud1991}. (We write
$W$ for the interval length, which plays the role of our $w$; the variable is
called $n$ in~\cite{Pikhurko2006}, which here denotes the order of the wheel.)

The sharp form \eqref{eq:rho-sharp} of the present bound has Pikhurko's constant
as its leading coefficient; the hypothesis of his theorem is satisfied by the sets
considered here, since at most $\ell-1=o(\ell^2)$ pairs are exempt from the
distinctness requirement. Its asymptotic content is $\ell\le\bigl(2/\sqrt{1+\lambda^2}\bigr)
\sqrt w\,(1+o(1))$, and letting $\rho_{w-1}\to1$ gives $\lambda\to2/(\pi+2)$ and
\[
  \frac2{\sqrt{1+4(\pi+2)^{-2}}}
  =\Bigl(\tfrac14+(\pi+2)^{-2}\Bigr)^{-1/2}
  =1.8639491169\ldots,
\]
which is Pikhurko's constant. The larger constant $\tfrac{15}8=1.875$ of
\eqref{eq:rho-closed} is the cost of two further relaxations: replacing
$\lambda_w$ by $\tfrac38$ in the passage from \eqref{eq:rho-sharp} to
\eqref{eq:rho-quadratic} costs $16/\sqrt{73}-1.86394\ldots=0.0087\ldots$, and
rounding $16/\sqrt{73}=1.87265\ldots$ up to the value $\tfrac{15}8$
costs a further $0.0023\ldots$; the estimate $\pi<22/7$ used to define $\lambda_w$
costs only $6.1\cdot10^{-5}$. What is needed below is not the size of the constant
but the validity of \eqref{eq:rho-closed} at every $w\ge1$, with no error term.
\end{remark}

\section{The geodesic support polynomial and the localization of large spokes}\label{sec:support-localization}

From this section on we assume that $W_{m+1}$ is geodesic Leech and set
\[
  N=\frac{m(m+3)}2,
  \qquad
  M=N+1=\frac{(m+1)(m+2)}2.
\]
The spoke labels in cyclic order are $a_0,\ldots,a_{m-1}$.

\subsection{The geodesic support polynomial}

Define
\[
  P(z)=\sum_i z^{a_i},
  \qquad
  F(z)=1+P(z),
\]
\[
  C(z)=\sum_i z^{a_i+a_{i+1}},
\]
\[
  Q(z)=1+P(z)+
  \sum_{\{i,j\}\notin E(C_m)}z^{a_i+a_j}.
\]
We call $C$ the \emph{adjacent-sum polynomial} and $Q$ the \emph{geodesic support polynomial}: its support consists of the weight $0$, all spoke labels, and all nonadjacent spoke sums, that is, of $0$ together with the weights of the two geodesic classes of Proposition~\ref{prop:wheel-geodesics} that involve the hub. These are pairwise distinct and lie in $\{0,1,\ldots,N\}$, so
\[
  |\operatorname{supp}Q|
  =1+m+\frac{m(m-3)}2=M-2m.
\]
These polynomials are related by the following identity.

\begin{lemma}[Geodesic support identity]\label{lem:support-identity}
We have
\begin{equation}\label{eq:support-identity}
  F(z)^2=2Q(z)+F(z^2)+2C(z)-2.
\end{equation}
\end{lemma}

\begin{proof}
Expanding, $F(z)^2=1+2P(z)+P(z)^2$ and
$P(z)^2=P(z^2)+2\sum_{i<j}z^{a_i+a_j}$. Splitting the pairs $\{i,j\}$ into
rim-adjacent and rim-nonadjacent ones writes the last sum as
$C(z)+\bigl(Q(z)-1-P(z)\bigr)$, and substituting $P(z^2)=F(z^2)-1$ gives
\eqref{eq:support-identity}.
\end{proof}

Write $\zeta=\e^{2\pi i/M}$ and
\[
  F_t=F(\zeta^t),\qquad
  q_t=|Q(\zeta^t)|,\qquad
  D_t=|C(\zeta^t)-1|.
\]
At nontrivial frequencies, \eqref{eq:support-identity} gives
\begin{equation}\label{eq:F-recursion-main}
  |F_t|^2\le2q_t+2D_t+|F_{2t}|.
\end{equation}

By Proposition~\ref{thm:frame-completion}, the values missing from $Q$ are exactly the residual set of rim weights
\[
  T=\set{b_i}\dd\set{b_i+b_{i+1}},
  \qquad |T|=2m.
\]
Since $Q(z)+\sum_{y\in T}z^y=1+z+\cdots+z^N$, at the nontrivial $M$-th roots of unity
\begin{equation}\label{eq:q-basic}
  q_t=\abs*{\sum_{y\in T}\zeta^{ty}}\le2m.
\end{equation}
Moreover, $D_t\le m+1$ always holds.

\subsection{The case of two large spokes}

Suppose that there are two large spokes; by Lemma~\ref{lem:large-spokes} they are adjacent. Write the two large spoke labels as $H,H-d$ and set
\[
  H=N-Y,\qquad
  \kappa=N-2Y,\qquad
  p=\kappa-d-1.
\]
Since both labels exceed $N/2$,
\[
  1\le d<\frac\kappa2.
\]
Moreover,
\begin{equation}\label{eq:phase-index-p}
  2H-d\equiv p\pmod M.
\end{equation}
Write the cycle as
\[
  H,\ H-d,\ y_0,\ y_1,\ldots,y_{m-4},\ L.
\]
Consider the following two paths.
\[
  P^+=(y_0,y_1,\ldots,y_{m-4}),
  \qquad
  P^-=(y_1,\ldots,y_{m-4},L).
\]
Each path has $s=m-3$ vertices. On the rim, $H$ is adjacent only to $H-d$ and $L$, and no label of $P^+$ is one of these two, so $H+y$ is the weight of a hub-type geodesic and therefore $H+y\le N$ for every label $y$ of $P^+$, that is, $y\le N-H=Y$. Likewise $H-d$ is adjacent only to $H$ and $y_0$, and no label of $P^-$ is one of these two, so $(H-d)+y\le N$ for every label $y$ of $P^-$, that is, $y\le Y+d$.

First, counting only the labels of $P^+$ and the non-defect pair sums internal to it gives
\[
  2Y\ge (m-3)+\frac{(m-4)(m-5)}2,
\]
whence $\kappa\le5m-7$. Since $N-3(5m-7)=\tfrac12(m^2-27m+42)\ge0$ for every $m\ge26$, this bound gives $3\kappa\le N$, which is exactly $H=N-Y\le2Y$, and also $2Y=N-\kappa\ge2\kappa$, so that $d<\kappa/2<\kappa\le Y$ and hence $L\le Y+d\le2Y$. Since moreover $H-d<H$ and the labels of $P^+$ are at most $Y$, all $m$ spoke labels lie in $[1,2Y]$, which is what is counted in (i) below. Consequently, when $m\ge40$ the occupancy terms appearing below are all smaller than $s$, since $d<\kappa\le5m-7$ and \eqref{eq:rho-monotone} and \eqref{eq:rho-closed} bound them by $\tfrac{15}8\sqrt{5m-7}+\tfrac{12}5<m-3=s$; so the counts in (iii)--(v) below are nonnegative as written and no truncation at zero is needed.

The following pairwise distinct geodesic weights are forced into $[1,2Y]$.
\begin{enumerate}[(i)]
\item all $m$ spoke labels;
\item the $(m-4)(m-5)/2$ non-defect pair sums internal to $P^+$;
\item at least $s-1-\mu(d)$ of the hub-type sums between $L$ and $P^+$;
\item at least $s-\mu(\kappa)$ of the sums between $H$ and $P^+$;
\item at least $s-\mu(\kappa)$ of the sums between $H-d$ and $P^-$;
\item at least $\lceil m/2\rceil$ of the rim labels.
\end{enumerate}
In (iii), $L\le Y+d$, and the labels of $P^+$ exceeding $Y-d$ lie in an interval of length $d$. In (iv), $H+x\le2Y$ is equivalent to $x\le3Y-N$, and the upper tail of $P^+$ has length $N-2Y=\kappa$. In (v), the condition $(H-d)+x\le2Y$ is instead equivalent to $x\le3Y-N+d$; since the labels of $P^-$ are bounded by $Y+d$ rather than by $Y$, the excluded tail $(3Y-N+d,\,Y+d]$ again contains exactly $\kappa$ integers, so the count is the same as in (iv). Item (vi) holds because if two rim labels exceeding $2Y$ were adjacent, their sum would exceed $4Y=2N-2\kappa>N$, the last inequality by the bound $\kappa\le5m-7$ above.

Summing (i)--(vi) gives
\[
  m+\frac{(m-4)(m-5)}2+\bigl(m-4-\mu(d)\bigr)+2\bigl(m-3-\mu(\kappa)\bigr)+\ceil*{\frac m2}
  =\frac{m^2}2-\frac m2+\ceil*{\frac m2}-\mu(d)-2\mu(\kappa),
\]
and all these values are pairwise distinct elements of $[1,2Y]$, so the right-hand side is at most $2Y=N-\kappa=\tfrac{m^2}2+\tfrac{3m}2-\kappa$. Rearranging,
\[
  \kappa\le\frac{3m}2-\left(\ceil*{\frac m2}-\frac m2\right)+\mu(d)+2\mu(\kappa)
  =2m-\ceil*{\frac m2}+\mu(d)+2\mu(\kappa)
  =\floor*{\frac{3m}2}+\mu(d)+2\mu(\kappa),
\]
the last equality holding in both parities, since $2m-\ceil{m/2}$ equals $3m/2$ for even $m$ and $(3m-1)/2$ for odd $m$. This yields the following.

\begin{lemma}[Localization of two large spokes]\label{lem:two-large-position}
If $m\ge40$, then
\begin{equation}\label{eq:z-fixedpoint}
  \kappa\le\floor*{\frac{3m}{2}}+\mu(d)+2\mu(\kappa),
  \qquad 1\le d<\frac\kappa2.
\end{equation}
Moreover
\begin{equation}\label{eq:z-3m}
  \kappa\le3m,
\end{equation}
and, setting $\Pi_m=\floor{5m/2}+1$,
\begin{equation}\label{eq:p-Pm}
  p=\kappa-d-1\le \Pi_m.
\end{equation}
\end{lemma}

\begin{proof}
Inequality~\eqref{eq:z-fixedpoint} follows directly from the count above.

We first prove \eqref{eq:z-3m}. For the boundary values $m=40,41,42$ we iterate
\eqref{eq:z-fixedpoint} with the difference $d$ held fixed: starting from
$\kappa\le5m-7$, we repeatedly substitute the current bound into the
right-hand side. Every term of the resulting sequence is an upper bound for
$\kappa$, because $\mu$ is nondecreasing by \eqref{eq:rho-monotone}; in every
case the sequence is nonincreasing and becomes constant after at most four
steps, and we write $\kappa(d)$ for its final value. Here $\mu(\kappa)$ is evaluated from
\eqref{eq:rho-sharp}, whose hypothesis $\kappa\ge10$ is satisfied throughout, and
$\mu(d)$ from the elementary bound \eqref{eq:rho-elementary}, which is what
governs the small values of $d$: the values of $d$ at which $\kappa(d)-d-1$ is
maximized below all lie in $1\le d\le9$, where \eqref{eq:rho-small} gives
$\mu(d)\le\varepsilon(d)$. For $d\ge10$ the
smaller of \eqref{eq:rho-elementary} and \eqref{eq:rho-sharp} is used. A value of
$d$ is admissible only while the accompanying
constraint $d<\kappa/2$ of \eqref{eq:z-fixedpoint} is still satisfiable, that is,
while $2d<\kappa(d)$; discarding the rest excludes the large values of $d$.
Taking the maxima over the admissible $d$ gives the values recorded in
Table~\ref{tab:kappa-maxima}.

\begin{table}[ht]
\centering
\caption{Maxima of $\kappa(d)$ and of $\kappa(d)-d-1$ over the admissible
$d$, at the three boundary values $m=40,41,42$.}\label{tab:kappa-maxima}
\begin{tabular}{ccccc}
\toprule
$m$ & $\max_d\kappa(d)$ & $3m$ & $\max_d\bigl(\kappa(d)-d-1\bigr)$ & $\Pi_m$\\
\midrule
$40$ & $120$ & $120$ & $101$ & $101$\\
$41$ & $123$ & $123$ & $102$ & $103$\\
$42$ & $126$ & $126$ & $106$ & $106$\\
\bottomrule
\end{tabular}
\end{table}

Both \eqref{eq:z-3m} and \eqref{eq:p-Pm} therefore hold at these three values
of $m$. The values of $d$ at which the two maxima are attained, and the full
iterate sequences at those values, are recorded in
Table~\ref{tab:kappa-iterates} of Appendix~\ref{app:kappa-iteration}.

For $m\ge43$, assume $\kappa\ge3m+1$. Then \eqref{eq:rho-closed}, which by
Proposition~\ref{lem:rho-closed} is available at every argument and in particular at
the possibly small $d$, together with $d<\kappa/2$, gives
\[
  \kappa<\frac{3m}{2}
  +\frac{15}{8}\left(2+\frac1{\sqrt2}\right)\sqrt \kappa
  +\frac{36}{5}.
\]
Using $1/\sqrt2<71/100$, the coefficient of the square root is less than $813/160$, so that $\kappa-\tfrac{813}{160}\sqrt\kappa<\tfrac{3m}2+\tfrac{36}5$. The map $\kappa\mapsto\kappa-\tfrac{813}{160}\sqrt\kappa$ has derivative $1-\tfrac{813}{320}\kappa^{-1/2}$ and is therefore increasing for $\kappa\ge(813/320)^2=660969/102400$, which holds here since $\kappa\ge3m+1\ge130$; it consequently suffices to substitute $\kappa=3m+1$, which gives $\tfrac{3m}2-\tfrac{31}5<\tfrac{813}{160}\sqrt{3m+1}$, both sides being positive for $m\ge43$. But
\[
  \left(\frac{3m}{2}-\frac{31}{5}\right)^2
  -\left(\frac{813}{160}\right)^2(3m+1)
  =\frac{57600m^2-2459067m+323095}{25600}>0,
\]
which is positive at $m=43$, and the forward difference
\[
  \frac{3(38400m-800489)}{25600}
\]
is also positive. This contradiction proves \eqref{eq:z-3m}.

Recall that $\Pi_m=\floor{5m/2}+1$. Substituting suitable integers into \eqref{eq:rho-quadratic} shows that for $m\ge42$,
\begin{equation}\label{eq:rho-two-aux}
  2\mu(3m)\le m+4,
  \qquad
  2\mu(\Pi_m+10)\le m+2.
\end{equation}
Concretely, writing $m=2k$ or $2k+1$ and substituting $\ell=k+3$ or $k+2$, the required inequalities reduce to the positivity of
\[
  73k^2-1442k+9,
  \quad
  73k^2-1442k-759,
\]
\[
  73k^2-1332k-2828,
  \quad
  73k^2-1332k-3340,
\]
all of which are increasing and positive for $k\ge21$.

If $d\ge8$, then \eqref{eq:rho-elementary} gives $\mu(d)\le\varepsilon(d)\le d-2$, so \eqref{eq:z-fixedpoint}, \eqref{eq:z-3m}, and \eqref{eq:rho-two-aux} give
\[
  p\le\floor*{\frac{3m}{2}}-3+2\mu(3m)\le \Pi_m.
\]
If $d\le7$, then $\mu(d)\le\varepsilon(d)\le d$ by \eqref{eq:rho-small}, so $p\le\floor{3m/2}-1+2\mu(\kappa)\le\Pi_m+2$, hence $\kappa=p+d+1\le \Pi_m+10$. Applying \eqref{eq:rho-two-aux} again gives $p\le \Pi_m$. The boundary values $m=40,41,42$ are covered by Table~\ref{tab:kappa-maxima}.
\end{proof}

\begin{remark}\label{rem:p-tight}
At $m=40$ and $m=42$ the maximum of $\kappa(d)-d-1$ equals $\Pi_m$ exactly, so
the iteration of \eqref{eq:z-fixedpoint} yields nothing stronger than
\eqref{eq:p-Pm} at those two values; only at $m=41$ is there slack. The value at
$m=40$ is used without slack in the case $(m,h)=(40,2)$ of
Proposition~\ref{lem:boundary}, as is made precise at the end of
Appendix~\ref{app:kappa-iteration}.
\end{remark}

\subsection{The case of one large spoke}

Let the unique large spoke be $H=N-Y$ and set
\[
  \kappa=N-2Y,
  \qquad
  e=2H-M=\kappa-1;
\]
we call $e$ the \emph{excess exponent} of $H$. Write the cycle as
\[
  H,\ y_0,\ y_1,\ldots,y_{m-3},\ L.
\]
Counting the core and its two flanking spokes in the same way as before yields the following.

\begin{lemma}[Localization of one large spoke]\label{lem:one-large-position}
If $m\ge40$, then
\begin{equation}\label{eq:one-large-fixedpoint}
  \kappa\le\floor*{\frac{3m}{2}}+1
  +\mu(\kappa)+2\mu(\floor{\kappa/2}).
\end{equation}
In particular,
\begin{equation}\label{eq:e-3m}
  e=\kappa-1\le3m-6.
\end{equation}
\end{lemma}

\begin{proof}
On the rim, $H$ is adjacent only to $y_0$ and $L$; we call these two spokes the
\emph{flanking spokes} and the path
\[
  (y_1,y_2,\ldots,y_{m-3})
\]
on the remaining $s=m-3$ rim vertices the \emph{core}. No label of the core is
one of the two flanking spokes, so $H+y$ is the weight of a hub-type geodesic and
therefore $H+y\le N$ for every core label $y$, that is, $y\le N-H=Y$. Since $H$
is the only spoke exceeding $N/2$, each flanking spoke is at most
$\floor{N/2}=Y+\floor{\kappa/2}$.

Counting only the core labels and the non-defect pair sums internal to the core
gives $2Y\ge(m-3)+(m-4)(m-5)/2$ exactly as before
Lemma~\ref{lem:two-large-position}, whence $\kappa\le5m-7$, $3\kappa\le N$ for
$m\ge26$, and $H=N-Y\le2Y$. Moreover $4Y=2N-2\kappa>N$, so that each flanking
spoke is at most $\floor{N/2}<2Y$; hence all $m$ spoke labels lie in $[1,2Y]$,
which is what is counted in (i) below, and, as before, for $m\ge40$ the
occupancy terms appearing below are all smaller than $s$, so the counts in
(iii)--(v) are nonnegative as written.

The following pairwise distinct geodesic weights are forced into $[1,2Y]$.
\begin{enumerate}[(i)]
\item all $m$ spoke labels;
\item the $(m-4)(m-5)/2$ non-defect pair sums internal to the core;
\item at least $s-\mu(\kappa)$ of the sums between $H$ and the core;
\item at least $s-1-\mu(\floor{\kappa/2})$ of the hub-type sums between $y_0$ and the core;
\item at least $s-1-\mu(\floor{\kappa/2})$ of the hub-type sums between $L$ and the core;
\item at least $\lceil m/2\rceil$ of the rim labels.
\end{enumerate}
In (iii), $H+x\le2Y$ is equivalent to $x\le3Y-N$, and the upper tail
$(3Y-N,\,Y]$ of the core contains exactly $N-2Y=\kappa$ integers. In (iv) and
(v), a flanking spoke $x'$ satisfies $x'+x\le2Y$ as soon as
$x\le Y-\floor{\kappa/2}$, and the corresponding tail
$(Y-\floor{\kappa/2},\,Y]$ contains $\floor{\kappa/2}$ integers; the further
subtraction of $1$ discards the one core vertex rim-adjacent to the flanking
spoke, namely $y_1$ for $y_0$ and $y_{m-3}$ for $L$, whose sum with it is not the
weight of a hub-type geodesic. Item (vi) holds because if two rim labels
exceeding $2Y$ were adjacent, their sum would exceed $4Y>N$.

Summing (i)--(vi) gives
\begin{align*}
  m&+\frac{(m-4)(m-5)}2+\bigl(m-3-\mu(\kappa)\bigr)
  +2\bigl(m-4-\mu(\floor{\kappa/2})\bigr)+\ceil*{\frac m2}\\
  &=\frac{m^2}2-\frac m2-1+\ceil*{\frac m2}-\mu(\kappa)-2\mu(\floor{\kappa/2}),
\end{align*}
and all these values are pairwise distinct elements of $[1,2Y]$, so the
right-hand side is at most $2Y=N-\kappa=\tfrac{m^2}2+\tfrac{3m}2-\kappa$.
Rearranging,
\[
  \kappa\le\frac{3m}2+1-\left(\ceil*{\frac m2}-\frac m2\right)
  +\mu(\kappa)+2\mu(\floor{\kappa/2})
  =2m-\ceil*{\frac m2}+1+\mu(\kappa)+2\mu(\floor{\kappa/2}),
\]
which is \eqref{eq:one-large-fixedpoint}, since $2m-\ceil{m/2}=\floor{3m/2}$ in
both parities, as noted before Lemma~\ref{lem:two-large-position}.

For $m=40,41$ we iterate \eqref{eq:one-large-fixedpoint} exactly as in the proof
of Lemma~\ref{lem:two-large-position}, starting from $\kappa\le5m-7$. Here $\mu$ is
evaluated from \eqref{eq:rho-sharp}, whose hypothesis is satisfied at every
argument occurring, the smallest being $\floor{115/2}=57$. The resulting
nonincreasing sequences are
\[
  193,\ 129,\ 118,\ 115
  \qquad\text{and}\qquad
  198,\ 130,\ 119,\ 116,
\]
so $\kappa\le115$ and $\kappa\le116$ respectively, both at most $3m-5$. The first
sequence stops at a fixed point of \eqref{eq:one-large-fixedpoint}, since
$\mu(115)\le22$ and $\mu(57)\le16$ give $60+1+22+2\cdot16=115$.

For $m\ge42$, assume $\kappa\ge3m-4$. Then \eqref{eq:rho-closed}, applied at
$\kappa$ and at $\floor{\kappa/2}\le\kappa/2$, gives
\[
  \kappa<\frac{3m}{2}
  +\frac{15}{8}\left(1+\frac2{\sqrt2}\right)\sqrt\kappa
  +\frac{41}{5}.
\]
Using $1/\sqrt2<99/140$, the coefficient of the square root is at most $507/112$,
so that $\kappa-\tfrac{507}{112}\sqrt\kappa<\tfrac{3m}2+\tfrac{41}5$. The map
$\kappa\mapsto\kappa-\tfrac{507}{112}\sqrt\kappa$ has derivative
$1-\tfrac{507}{224}\kappa^{-1/2}$ and is therefore increasing for
$\kappa\ge(507/224)^2=257049/50176$, which holds here since
$\kappa\ge3m-4\ge122$; it consequently suffices to substitute $\kappa=3m-4$,
which gives $\tfrac{3m}2-\tfrac{61}5<\tfrac{507}{112}\sqrt{3m-4}$, both sides
being positive for $m\ge42$. But
\[
  \left(\frac{3m}{2}-\frac{61}{5}\right)^2
  -\left(\frac{507}{112}\right)^2(3m-4)
  =\frac{705600m^2-30756435m+72381124}{313600}>0,
\]
which is positive at $m=42$, and the forward difference
\[
  \frac{3(470400m-10016945)}{313600}
\]
is also positive. This contradiction proves $\kappa\le3m-5$, that is,
$e\le3m-6$.
\end{proof}

\subsection{Two low frequencies of the adjacent-sum polynomial}

In the case of two large spokes, let $L$ be the spoke adjacent to $H$ other than $H-d$, in the notation introduced before Lemma~\ref{lem:two-large-position}. Using Lemma~\ref{lem:two-large-position}, the first two frequencies of the adjacent-sum polynomial can be controlled uniformly.

\begin{lemma}[Low-frequency adjacent-sum bounds]\label{lem:D12}
If $m\ge40$ and there are two large spokes, then
\begin{equation}\label{eq:D1}
  D_1=|C(\zeta)-1|\le m-\frac34,
\end{equation}
\begin{equation}\label{eq:D2}
  D_2=|C(\zeta^2)-1|\le m+\frac12.
\end{equation}
\end{lemma}

\begin{proof}
Write $\alpha=2\pi p/M$, where $p=\kappa-d-1=2H-d-M$ as in
\eqref{eq:phase-index-p}, so that $\zeta^{2H-d}=\e^{i\alpha}$. By
\eqref{eq:p-Pm} we have $p\le\Pi_m$, and $8\pi\Pi_m<3M$ for $m\ge40$, so
$0\le\alpha<\tfrac34$.

First,
\[
  D_1\le m-2+
  |\zeta^{2H-d}+\zeta^{H+L}-1|.
\]
If $H+L\le M$, write $\zeta^{H+L}=\e^{-i\beta}$ with
$\beta=2\pi(M-H-L)/M\ge0$. Since $p=2H-d-M$ and $d,L\ge1$, we obtain
$p<2H+2L-M$, which is exactly $\beta+\tfrac\alpha2<\pi$. Expanding,
\[
  |\e^{i\alpha}+\e^{-i\beta}-1|^2
  =3+2\cos(\alpha+\beta)-2\cos\alpha-2\cos\beta
  =3-2\cos\alpha-4\sin\frac\alpha2\,\sin\Bigl(\beta+\frac\alpha2\Bigr).
\]
Since $0\le\beta+\alpha/2<\pi$, the last sine is nonnegative, and
$\cos\alpha\ge1-\alpha^2/2$ gives
\[
  |\e^{i\alpha}+\e^{-i\beta}-1|^2
  \le3-2\cos\alpha\le1+\alpha^2\le\frac{25}{16}.
\]
If $H+L\ge M$, write $\zeta^{H+L}=\e^{i\gamma}$ with $\gamma=2\pi\delta/M$,
where $\delta=H+L-M\ge0$. Since $L\le Y+d$, we have
$p+\delta=\kappa-d+L-Y-2\le\kappa-2\le3m-2$, the last inequality by
\eqref{eq:z-3m}. Set $\theta=(\alpha+\gamma)/2=\pi(p+\delta)/M$ and
$u=(\alpha-\gamma)/2$. Expanding,
\[
  |\e^{i\alpha}+\e^{i\gamma}-1|^2
  =3+2\cos2u-4\cos\theta\cos u
  =1+4\cos^2u-4\cos\theta\cos u,
\]
which, as a quadratic in $\cos u\in[\cos\theta,1]$, is increasing there and
hence maximized at $\cos u=1$, that is, when the two phases coincide. Thus
\[
  |\e^{i\alpha}+\e^{i\gamma}-1|^2
  \le5-4\cos\theta
  \le1+2\theta^2
  <\frac{25}{16},
\]
where the last step uses $\theta\le\pi(3m-2)/M<\tfrac12$ for $m\ge40$. Hence
\eqref{eq:D1} holds.

Second,
\[
  D_2\le m-1+|\e^{2i\alpha}-1|
  =m-1+2\sin\alpha.
\]
Since $\sin\alpha<\alpha<\tfrac34$, the last term is less than $\tfrac32$.
This gives \eqref{eq:D2}.
\end{proof}

\section{Uniform Fourier inequalities and the elimination of the infinite range}\label{sec:uniform-fourier}

\subsection{A lower bound for the finite Fourier kernel}

In this section we set
\[
  R=\floor*{\frac{M-1}{2}},
  \qquad
  \rho_R=\sum_{r=1}^{R}c_r=1-\frac1{2R+1},
\]
with $c_r=2/(4r^2-1)$ and the finite kernel $K_R$ as in
Section~\ref{sec:fourier}. Lemma~\ref{lem:kernel} gives
\[
  K_R(x)\ge \rho_R\quad(0\le x\le\pi),
\]
\[
  K_R(x)\ge \rho_R+\pi\sin x\quad(\pi\le x\le2\pi).
\]
Hence, setting
\[
  \Lambda=
  \sum_{a\in\{0,a_0,\ldots,a_{m-1}\}}
  K_R\!\left(\frac{2\pi a}{M}\right),
\]
we obtain
\begin{equation}\label{eq:Lambda-lower-general}
  \Lambda\ge \rho_R(m+1)+
  \pi\sum_{a_i>N/2}\sin\frac{2\pi a_i}{M}.
\end{equation}
On the other hand, the Fourier expansion is
\begin{equation}\label{eq:Lambda-fourier}
  \Lambda=\frac\pi2\operatorname{Im}F_1+
  \sum_{r=1}^{R}c_r\operatorname{Re}F_{2r}.
\end{equation}
Moreover,
\begin{equation}\label{eq:coeff-tail-third}
  \rho_R\ge1-\frac1N,
  \qquad
  \sum_{r=2}^{R}c_r<\frac13.
\end{equation}

\subsection{Phase coupling of the first two Fourier terms}

\begin{lemma}[Phase-pair inequality]\label{lem:phase-pair}
Let $\alpha,\beta,\Gamma,V>0$ with $\Gamma>V$, and set
\[
  \tau=\frac{\alpha}{2\sqrt{\Gamma-V}}.
\]
Assume also
\[
  0<\tau<\beta,
  \qquad
  \Gamma(\beta-\tau)\ge V(2\beta-\tau).
\]
If complex numbers $z,w$ satisfy
\[
  |z^2-w|\le \Gamma,
  \qquad
  |w|\le V,
\]
then
\begin{equation}\label{eq:phase-pair}
  \alpha\operatorname{Im}z+
  \beta\operatorname{Re}w
  \le
  \alpha\sqrt{\Gamma-V}+\beta V.
\end{equation}
\end{lemma}

\begin{proof}
Let $q=z^2$. If $\operatorname{Im}z<0$ the left-hand side decreases, so we may use
\[
  \operatorname{Im}z
  \le\sqrt{\frac{|q|-\operatorname{Re}q}{2}}.
\]
Moreover,
\[
  |q|-\operatorname{Re}q
  =\max_{|\xi|\le1}\operatorname{Re}((\xi-1)q).
\]
Using Young's inequality
\[
  \alpha\sqrt u\le\tau u+\frac{\alpha^2}{4\tau}
  \qquad(u\ge0)
\]
and $q=w+\Delta$ with $|\Delta|\le \Gamma$, we obtain
\[
\begin{aligned}
 \alpha\operatorname{Im}z+\beta\operatorname{Re}w
 \le{}&\frac{\alpha^2}{4\tau}
 +\max_{|\xi|\le1}
 \left[
   \frac{\tau \Gamma}{2}|\xi-1|
   +V\abs*{\beta+\frac\tau2(\xi-1)}
 \right].
\end{aligned}
\]
Let $\varrho=|\xi-1|$. For $|\xi|\le1$ we have
$\operatorname{Re}(\xi-1)\le-\varrho^2/2$, so the bracket is at most
\[
  f(\varrho)=\frac{\tau \Gamma}{2}\varrho+
  V\sqrt{\beta^2-
  \frac{\tau(2\beta-\tau)}4\varrho^2}
  \qquad(0\le \varrho\le2).
\]
Write $k=\tau(2\beta-\tau)/4>0$; then
$f'(\varrho)=\tfrac{\tau \Gamma}2-Vk\varrho(\beta^2-k\varrho^2)^{-1/2}$ is decreasing on $[0,2]$,
because $k\varrho$ and $(\beta^2-k\varrho^2)^{-1/2}$ both increase there, so $f'$ attains
its minimum on $[0,2]$ at $\varrho=2$. Since $\beta^2-4k=(\beta-\tau)^2$ and
$\tau<\beta$, we have
$f'(2)=\tau\bigl(\Gamma(\beta-\tau)-V(2\beta-\tau)\bigr)/\bigl(2(\beta-\tau)\bigr)$,
which is nonnegative by hypothesis; hence $f'\ge0$ on $[0,2]$ and the maximum
is attained at $\varrho=2$, where $\sqrt{\beta^2-4k}=\beta-\tau$. Therefore
\[
  \alpha\operatorname{Im}z+\beta\operatorname{Re}w
  \le\frac{\alpha^2}{4\tau}+\tau \Gamma+V(\beta-\tau),
\]
and substituting the specified $\tau$ makes the right-hand side equal to the right-hand side of \eqref{eq:phase-pair}.
\end{proof}

From now on set
\[
  \pi_+=\frac{355}{113},
  \qquad
  u_m=\sqrt{7m+3}.
\]
From \eqref{eq:F-recursion-main}, \eqref{eq:q-basic}, $D_t\le m+1$, and $|F_{2t}|\le m+1$, we obtain, for every nontrivial frequency,
\begin{equation}\label{eq:u-bound}
  |F_t|\le u_m.
\end{equation}

\subsection{The case of no large spoke}\label{ssec:h0}

Let
\[
  v_m^{(0)}=\sqrt{6m+2+u_m}
\]
and define
\begin{equation}\label{eq:U0}
  \mathcal U_0(m)=
  \frac{\pi_+}{2}\sqrt{6m+2-v_m^{(0)}}
  +\frac23v_m^{(0)}+\frac13u_m.
\end{equation}
Applying Lemma~\ref{lem:phase-pair} to the expansion \eqref{eq:Lambda-fourier} with $z=F_1$, $w=F_2$, $\Gamma=6m+2$, $V=v_m^{(0)}$, $\alpha=\pi/2$, $\beta=2/3$, and bounding the tail $\sum_{r\ge2}$ of \eqref{eq:Lambda-fourier} by \eqref{eq:coeff-tail-third} and \eqref{eq:u-bound}, we obtain
\[
  \Lambda\le\mathcal U_0(m).
\]
Here $F_1^2-F_2=2Q(\zeta)+2\bigl(C(\zeta)-1\bigr)$ by the support identity \eqref{eq:support-identity}, so $|F_1^2-F_2|\le2q_1+2D_1\le6m+2=\Gamma$ by \eqref{eq:q-basic} and $D_t\le m+1$; and $|F_2|^2\le6m+2+u_m$ by \eqref{eq:F-recursion-main} and \eqref{eq:u-bound}, so the hypotheses $|z^2-w|\le \Gamma$ and $|w|\le V$ hold. The remaining two hypotheses of Lemma~\ref{lem:phase-pair}, namely $\tau<\beta$ and $\Gamma(\beta-\tau)\ge V(2\beta-\tau)$, hold uniformly in $m$, because $\Gamma$ grows linearly in $m$ while $V$ is of order $\sqrt \Gamma$. It suffices to record the following elementary implication, in which $\alpha=\pi/2$ and $\beta=2/3$ are the values fixed above:
\begin{equation}\label{eq:phase-pair-uniform}
  \Gamma\ge242,\quad V^2\le\tfrac32\Gamma
  \qquad\Longrightarrow\qquad
  \tau<\tfrac1{16}<\beta,\quad \Gamma(\beta-\tau)>V(2\beta-\tau).
\end{equation}
Indeed, $\Gamma\mapsto \Gamma-\sqrt{3\Gamma/2}$ is increasing, and $\sqrt{363}<20$, so
\[
  \Gamma-V\ge \Gamma-\sqrt{\tfrac32\Gamma}\ge242-\sqrt{363}>222>196 ;
\]
hence $\sqrt{\Gamma-V}>14$ and $\tau=\pi/(4\sqrt{\Gamma-V})<\pi/56<\tfrac1{16}$. Consequently $\beta-\tau>\tfrac{29}{48}$ and $2\beta-\tau<\tfrac43$, so the second conclusion follows from $\tfrac{29}{48}\Gamma\ge\tfrac43\sqrt{\tfrac32\Gamma}$, which upon squaring both sides reads $\Gamma\ge\tfrac{6144}{841}$.

Both upper bounds obtained in this section satisfy the hypotheses of \eqref{eq:phase-pair-uniform}. For $\mathcal U_0$, which is applied for $m\ge40$, we have $\Gamma=6m+2\ge242$ and $V^2=\Gamma+u_m$, and $u_m\le\tfrac\Gamma2=3m+1$ reads $9m^2-m-2\ge0$ after squaring. For $\mathcal U_2$ below, which is applied for $m\ge42$, we have $\Gamma=6m-\tfrac32\ge249$ and $V^2=\Gamma+\tfrac52+u_m$, and $u_m\le3m-\tfrac{13}4$ reads $9m^2-\tfrac{53}2m+\tfrac{121}{16}\ge0$ after squaring, a quadratic whose larger root is smaller than $3$. The applications in Section~\ref{sec:boundary}, where $\Gamma$ is of size $2D_1$ rather than $6m$, are verified separately there.

If there is no large spoke, then \eqref{eq:Lambda-lower-general} gives
\[
  \Lambda\ge\mathcal L_0(m):=
  m+1-\frac{2(m+1)}{m(m+3)}.
\]
We follow the conventions of Section~\ref{sec:strategy} for the decimals. At $m=40$,
\[
  u_{40}=16.8226\ldots,
  \qquad
  v_{40}^{(0)}=16.0879\ldots,
\]
\[
  \mathcal U_0(40)<39.943,
  \qquad
  \mathcal L_0(40)>40.9523,
\]
so
\[
  \mathcal L_0(40)-\mathcal U_0(40)>1.009.
\]
Moreover, for $m\ge40$,
\[
  u_m'<\frac7{32},
  \qquad
  \bigl(v_m^{(0)}\bigr)'<\frac15,
  \qquad
  \left(\sqrt{6m+2-v_m^{(0)}}\right)'<\frac15,
\]
so
\[
  \mathcal U_0'(m)<\frac{\pi_+}{10}+\frac2{15}+\frac7{96}<0.54,
\]
while the fraction subtracted in $\mathcal L_0$ is decreasing, so $\mathcal L_0'(m)>1$. Hence we obtain a contradiction for all $m\ge40$.

\subsection{The case of one large spoke}\label{ssec:h1}

Let $H$ be the unique large spoke and let $e=2H-M$ be its excess exponent, as in Section~\ref{sec:support-localization};
this is nonnegative, since $2H>N=M-1$ and $H$ is an integer. Hence
$2\pi H/M=\pi+\pi e/M$, so that
$\sin(2\pi H/M)=-\sin(\pi e/M)\ge-\pi e/M$, and inserting this into
\eqref{eq:Lambda-lower-general}, together with $e\le3m-6$ from
\eqref{eq:e-3m}, $1/M=2/\bigl((m+1)(m+2)\bigr)$, and $\pi<\pi_+$, gives
\begin{equation}\label{eq:L1}
  \Lambda\ge\mathcal L_1(m):=
  \mathcal L_0(m)-
  \frac{2\pi_+^2(3m-6)}{(m+1)(m+2)}.
\end{equation}
The upper bound \eqref{eq:U0} still applies. At $m=41$,
\[
  u_{41}=17.0293\ldots,
  \qquad
  v_{41}^{(0)}=16.2797\ldots,
\]
\[
  \mathcal U_0(41)<40.4409,
  \qquad
  \mathcal L_1(41)>40.6746,
\]
so
\[
  \mathcal L_1(41)-\mathcal U_0(41)>0.233.
\]
The two rational functions subtracted in \eqref{eq:L1} are decreasing for $m\ge41$, so $\mathcal L_1'(m)>1$, and hence we obtain a contradiction for all $m\ge41$.

\subsection{The case of two large spokes}\label{ssec:h2}

Let
\[
  v_m^{(2)}=\sqrt{6m+1+u_m}
\]
and define
\[
  \mathcal U_2(m)=
  \frac{\pi_+}{2}\sqrt{6m-\frac32-v_m^{(2)}}
  +\frac23v_m^{(2)}+\frac13u_m.
\]
From \eqref{eq:D1}, \eqref{eq:D2}, and Lemma~\ref{lem:phase-pair},
\[
  \Lambda\le\mathcal U_2(m).
\]
The excess exponents $e_1=2H-M$ and $e_2=2(H-d)-M$ of the two large spokes
are nonnegative, sum to $2p$ by \eqref{eq:phase-index-p}, and satisfy
$e_2\le e_1=\kappa-1\le3m-1<M$ by \eqref{eq:z-3m}. As in
Section~\ref{ssec:h1}, the two sines in \eqref{eq:Lambda-lower-general} equal
$-\sin(\pi e_1/M)$ and $-\sin(\pi e_2/M)$, and since $\pi e_1/M$ and
$\pi e_2/M$ lie in $[0,\pi]$, the concavity of the sine there gives
$\sin(\pi e_1/M)+\sin(\pi e_2/M)\le2\sin(\pi p/M)\le2\pi p/M$. With
$p\le\Pi_m\le(5m+2)/2$ from \eqref{eq:p-Pm} and
$1/M=2/\bigl((m+1)(m+2)\bigr)$, this yields
\begin{equation}\label{eq:L2}
  \Lambda\ge\mathcal L_2(m):=
  \mathcal L_0(m)-
  \frac{2\pi_+^2(5m+2)}{(m+1)(m+2)}.
\end{equation}
At $m=42$,
\[
  u_{42}=17.2336879\ldots,
  \qquad
  v_{42}^{(2)}=16.4387860\ldots,
\]
\[
  \mathcal U_2(42)<40.7355,
  \qquad
  \mathcal L_2(42)>40.7427.
\]
Hence $\mathcal L_2(42)-\mathcal U_2(42)>0.0072$.

Moreover, for $m\ge42$,
\[
  u_m'<\frac7{34},
  \qquad
  (v_m^{(2)})'<\frac15,
  \qquad
  \left(\sqrt{6m-\frac32-v_m^{(2)}}\right)'<\frac15.
\]
Therefore
\[
  \mathcal U_2'(m)
  <\frac{\pi_+}{10}+\frac{2}{15}+\frac7{102}<1.
\]
On the other hand, the two rational functions subtracted in \eqref{eq:L2} are decreasing, so $\mathcal L_2'(m)>1$. Hence we obtain a contradiction for all $m\ge42$.

The results of this section are summarized as follows; recall that $h$ is the
number of large spokes and that $m=n-1$.

\begin{proposition}\label{prop:uniform-ranges}
In a geodesic Leech wheel $W_{m+1}$ with $h$ large spokes, the following hold.
\begin{enumerate}[(i)]
\item If $h=0$, then $m\le39$, that is, $n\le40$.
\item If $h=1$, then $m\le40$, that is, $n\le41$.
\item If $h=2$, then $m\le41$, that is, $n\le42$.
\end{enumerate}
\end{proposition}

\begin{proof}
Part~(i) is the conclusion of Section~\ref{ssec:h0}, part~(ii) that of
Section~\ref{ssec:h1}, and part~(iii) that of Section~\ref{ssec:h2}.
\end{proof}

\section{The six-variable Parseval argument for the last three boundary cases}\label{sec:boundary}

After Proposition~\ref{prop:uniform-ranges}, the only remaining cases are
\[
  (m,h)=(40,1),\quad(40,2),\quad(41,2),
\]
that is, the wheels $W_{41}$ and $W_{42}$. In this section we do not optimize all the Fourier coefficients individually; we keep only six coordinates. The six retained frequencies $1,2,4,6,8,12$ are exactly those at which $q_t$ occurs in the majorant \eqref{eq:Psi}: the frequencies $1$ and $2$ enter through Lemma~\ref{lem:phase-pair}, and the pairs $4,8$ and $6,12$ through the two-step bounds for $|F_4|$ and $|F_6|$. No further coordinate is produced by the tail, because the only $r\in\{4,\ldots,R\}$ whose frequencies $2r$ or $4r$ reduce, up to sign modulo $M$, into this list are the finitely many recorded after \eqref{eq:Y-tail}. Retaining further coordinates could be expected to widen the final margin, but this is not needed. The tangent points used below were located by numerically maximizing $\Psi$ over $\mathcal D$ and rounding the maximizer to two decimal places. By Lemma~\ref{lem:psi-concave} the tangent-plane bound \eqref{eq:tangent-master} is valid at every point of $\mathcal D$, so the argument does not depend on the maximizer having been found exactly: only the tabulated rational inequalities are used in the proof.

\subsection{The Parseval identity for the residual set}

For $m=40,41$ the modulus $M$ is odd. Recall from \eqref{eq:q-basic} that,
at the nontrivial frequencies,
\[
  q_t=\abs*{\sum_{y\in T}\zeta^{ty}}
\]
for the residual set $T$. Parseval's identity and conjugate symmetry give
\begin{equation}\label{eq:parseval-q}
  \sum_{t=1}^{(M-1)/2}q_t^2
  =E_m:=m(M-2m).
\end{equation}
Moreover, $t=1,2,4$ are coprime to $M$. The absolute value of a sum of $k$ distinct $M$-th roots of unity is largest
when the $k$ roots are consecutive. Indeed, for $J\subseteq\Z/M\Z$ with
$\abs J=k$ we have
$\abs*{\sum_{j\in J}\zeta^{j}}=\max_{\phi}\sum_{j\in J}\cos(2\pi j/M-\phi)$.
For each fixed $\phi$ the inner sum is maximal over the $k$-element
subsets when $J$ collects $k$ roots whose angles are nearest to $\phi$. Such
a set may be taken to consist of the roots lying in an arc centered at $\phi$,
and since the $M$-th roots are equally spaced, it is then a cyclically
consecutive block. Interchanging the two maxima, the largest value is therefore attained at
a consecutive block, for which the absolute value is
$\abs*{\sin(k\pi/M)}/\sin(\pi/M)$. Since $\abs T=2m$ and $0<2m<M$, this
gives
\begin{equation}\label{eq:primitive-char}
  q_1,q_2,q_4\le
  c_m^*:=\frac{\sin(2m\pi/M)}{\sin(\pi/M)}.
\end{equation}
Taylor inequalities and $103993/33102<\pi<355/113$ give
\begin{equation}\label{eq:cstar-bounds}
  c_{40}^*<78.87,
  \qquad
  c_{41}^*<80.90.
\end{equation}

Now set
\[
  x=(q_1,q_2,q_4,q_6,q_8,q_{12}),
  \qquad
  \eta(x)=\sqrt{E_m-\|x\|_2^2}.
\]
The elementary uniform bound obtained from \eqref{eq:F-recursion-main} is
\[
  U_m^*=\frac{1+\sqrt{24m+9}}2.
\]
Define the following two-step bound function.
\[
  \cV_m(y,y')=
  \sqrt{2y+2m+2+
  \sqrt{2y'+2m+2+U_m^*}}.
\]
Also set
\[
  R=\frac{M-1}{2},
  \qquad
  \omega_m=\sum_{r=4}^{R}c_r=\frac17-\frac1M.
\]
Let $\sigma_M$ be the unique integer with $1\le\sigma_M\le R$ and $4\sigma_M\equiv\pm1\pmod M$. For $M=861$ we have $\sigma_M=215$, and for $M=903$ we have $\sigma_M=226$.

Define
\begin{align}
 \cX_m(x)&=
 \frac{c_Rq_1+c_4q_8+c_6q_{12}+\frac{267}{10000}\eta(x)}{\omega_m},
 \label{eq:X-tail}\\
 \cY_m(x)&=
 \frac{c_{\sigma_M}q_1+c_Rq_2+c_{R-1}q_6+\frac{219}{5000}\eta(x)}{\omega_m}.
 \label{eq:Y-tail}
\end{align}
Here the remaining index sets are
\[
  \mathcal I_2=\{5\}\cup\{7,8,\ldots,R-1\},
  \qquad
  \mathcal I_4=\{4,5,\ldots,R\}\setminus\{\sigma_M,\,R-1,\,R\},
\]
namely those $r\in\{4,\ldots,R\}$ whose frequencies $2r$ (for $\cX_m$) and $4r$ (for $\cY_m$) do not reduce, up to sign modulo $M$, to one of the six retained frequencies $1,2,4,6,8,12$. A direct computation of the coefficient square sums gives, for both values of $M$,
\[
  \Bigl(\sum_{r\in\mathcal I_2}c_r^2\Bigr)^{1/2}<\frac{267}{10000},
  \qquad
  \Bigl(\sum_{r\in\mathcal I_4}c_r^2\Bigr)^{1/2}<\frac{219}{5000}.
\]
This is verified by adding the rational terms with $r\le20$ directly and bounding the tail via $c_r<2/(2r-1)^2$, which gives
\[
  \sum_{r\ge21}c_r^2
  <4\int_{20}^{\infty}\frac{du}{(2u-1)^4}
  =\frac{2}{3\cdot39^{3}}
  <\frac1{88000}.
\]
Therefore, by Cauchy--Schwarz and \eqref{eq:parseval-q}, the quantities \eqref{eq:X-tail} and \eqref{eq:Y-tail} are upper bounds for the $c_r$-weighted averages of the $q_{2r}$ and of the $q_{4r}$ appearing in the tails, respectively.

\subsection{The concave majorant}

According to the case, let $\Db_1,\Db_2$ be the following upper bounds for the
quantities $D_1,D_2$ of Section~\ref{sec:support-localization}:
\[
  (\Db_1,\Db_2)=
  \begin{cases}
  (m+1,m+1),&h=1,\\
  (m-\frac34,m+\frac12),&h=2.
  \end{cases}
\]
For $h=1$ these are the bounds $D_t\le m+1$ valid at every frequency; for $h=2$
they are Lemma~\ref{lem:D12}, whose hypotheses hold at $(40,2)$ and $(41,2)$.
Replacing $D_1,D_2$ by $\Db_1,\Db_2$ below is thus an inequality, and it is the
only place in this section where one is used implicitly.
Also set
\[
  v_2(x)=
  \sqrt{2q_2+2\Db_2+
  \sqrt{2q_4+2m+2+U_m^*}}.
\]
By \eqref{eq:F-recursion-main}, Lemma~\ref{lem:phase-pair}, Jensen's inequality, and \eqref{eq:X-tail}--\eqref{eq:Y-tail}, the entire Fourier right-hand side is bounded above by the following function.
\begin{align}
 \Psi_{m,\Db_1,\Db_2}(x)={}&
 \frac{\pi_+}{2}\sqrt{2q_1+2\Db_1-v_2(x)}
 +\frac23v_2(x)\notag\\
 &+c_2\cV_m(q_4,q_8)
 +c_3\cV_m(q_6,q_{12})
 +\omega_m\cV_m(\cX_m(x),\cY_m(x)).
 \label{eq:Psi}
\end{align}
Here $\Gamma$ is of size $2\Db_1$ rather than $6m$, so \eqref{eq:phase-pair-uniform} does not apply and the hypotheses of Lemma~\ref{lem:phase-pair} are verified directly instead: uniformly over $0\le q_1,q_2,q_4\le c_m^*$ one has $\tau<0.1<\beta$ and $\Gamma(\beta-\tau)>44>21>V(2\beta-\tau)$ in each of the three cases. (The extreme values are attained at $q_1=0$, $q_2=q_4=c_m^*$ for $\tau$ and for $\Gamma(\beta-\tau)$, and at $q_1=q_2=q_4=c_m^*$ for $V(2\beta-\tau)$.)

We work on the convex domain
\[
  \mathcal D=\set{x\in\R^6:\ x\ge0,\ \|x\|_2^2\le E_m},
\]
which contains the true vector of moduli and every tangent point used below;
on the ball alone the radicands of \eqref{eq:Psi} need not be positive.

\begin{lemma}[Concavity of the majorant]\label{lem:psi-concave}
The function $\Psi_{m,\Db_1,\Db_2}$ is concave on $\mathcal D$.
\end{lemma}

\begin{proof}
The last three summands are covered by the elementary composition rule: $\cV_m$ is
concave and nondecreasing in each argument, $\eta(x)=\sqrt{E_m-\|x\|_2^2}$ is
concave on $\mathcal D$, and $\cX_m,\cY_m$ are affine functions of $x$ plus a positive
multiple of $\eta$, hence concave; so $\cV_m(\cX_m,\cY_m)$, $\cV_m(q_4,q_8)$, and
$\cV_m(q_6,q_{12})$ are concave.

The first two summands must be treated together, because the radicand of the
first is an affine function \emph{minus} the concave function $v_2$, hence convex,
and the first summand is decreasing in $q_2$ and $q_4$. Introduce
\[
  \psi(q_1,s)=\frac{\pi_+}2\sqrt{2q_1+2\Db_1-s}+\frac23s ,
\]
so that the two summands equal $\psi\bigl(q_1,v_2(x)\bigr)$. Now
$(q_1,s)\mapsto2q_1+2\Db_1-s$ is affine, so its square root is jointly concave, and
adding the linear term $\tfrac23s$ leaves $\psi$ jointly concave in $(q_1,s)$.
Moreover, writing $\Theta=2q_1+2\Db_1-s$,
\[
  \frac{\partial\psi}{\partial s}=\frac23-\frac{\pi_+}{4\sqrt\Theta}\ge0
  \qquad\text{as soon as}\qquad
  \Theta\ge\Bigl(\frac{3\pi_+}8\Bigr)^2=1.387\ldots,
\]
so $\psi$ is nondecreasing in $s$ on the relevant range. This requirement is satisfied on all of $\mathcal D$:
on $\mathcal D$ one has $q_2,q_4\le\sqrt{E_m}$, whence
\[
  v_2(x)\le\sqrt{2\sqrt{E_m}+2\Db_2+\sqrt{2\sqrt{E_m}+2m+2+U_m^*}}<21.8
\]
and therefore $\Theta\ge2\Db_1-21.8>56$ in each of the three cases $(m,h)$, the smallest value being $\Theta>56.7$ at $(m,h)=(40,2)$, where $\Db_1=\tfrac{157}4$. Since $q_1$ is
an affine function of $x$ and $v_2$ is concave on $\mathcal D$, the composition of the
jointly concave, $s$-nondecreasing $\psi$ with $(q_1,v_2(x))$ is concave. Adding the
three concave summands above gives the claim.
\end{proof}

For any $x_0$ in the interior of $\mathcal D$, let $g=\nabla\Psi(x_0)$ and $A_0=\Psi(x_0)-g\cdot x_0$. By Lemma~\ref{lem:psi-concave},
\[
  \Psi(x)\le A_0+g\cdot x
  \qquad(x\in\mathcal D).
\]
Using \eqref{eq:primitive-char}, Cauchy--Schwarz, and the fact that the three
tangent points below all have $g_1,g_2,g_3>0$ (see Table~\ref{tab:app-gradients}), so that $q_i\le c_m^*$ may be substituted in the
first three coordinates,
\begin{equation}\label{eq:tangent-master}
  \Psi(x)
  \le A_0+c_m^*(g_1+g_2+g_3)
  +\sqrt{E_m}\,\|(g_4,g_5,g_6)\|_2.
\end{equation}

\subsection{Three explicit tangent planes}

The results of this section are collected in the following statement, which
is cited in the proof of Theorem~\ref{thm:main}.

\begin{proposition}[The three boundary cases]\label{lem:boundary}
For $(m,h)\in\{(40,1),(40,2),(41,2)\}$, no geodesic Leech labeling of $W_{m+1}$
has exactly $h$ large spokes; that is, the three boundary cases left open by
Proposition~\ref{prop:uniform-ranges} are impossible.
\end{proposition}

\begin{proof}
We use the following rational tangent points.
\[
\begin{aligned}
 x_{40,2}&=(78.86,78.86,78.86,75.64,60.64,27.20),\\
 x_{40,1}&=(78.86,78.86,78.86,75.61,60.66,27.24),\\
 x_{41,2}&=(80.89,80.89,80.89,79.75,64.12,28.80).
\end{aligned}
\]
Differentiating \eqref{eq:Psi} directly, and verifying upper and lower bounds for each square root by squaring both sides, yields the rational inequalities of Table~\ref{tab:tangent-bounds}.

\begin{table}[ht]
\centering
\caption{Rational bounds on the tangent-plane data at the three boundary
cases.}\label{tab:tangent-bounds}
\begin{tabular}{c@{\qquad}c@{\qquad}c@{\qquad}c@{\qquad}c}
\toprule
$(m,h)$ & $A_0<$ & $g_1+g_2+g_3<$ & $\|(g_4,g_5,g_6)\|_2<$ & $\Psi<$\\
\midrule
$(40,2)$ & $27.444$ & $0.14247$ & $0.00000359$ & $38.682$\\
$(40,1)$ & $27.717$ & $0.14159$ & $0.00000503$ & $38.886$\\
$(41,2)$ & $27.782$ & $0.14107$ & $0.00000399$ & $39.196$\\
\bottomrule
\end{tabular}
\end{table}

For example, the first row follows from \eqref{eq:tangent-master}, \eqref{eq:cstar-bounds}, and $\sqrt{31240}<177$ via
\[
  \Psi<27.444+78.87(0.14247)+177(0.00000359)<38.682.
\]
The other two rows use $\sqrt{31240}<177$ and $\sqrt{33661}<184$, respectively. All decimals in Table~\ref{tab:tangent-bounds} are terminating, hence exact rational numbers. Appendix~\ref{app:tangent} records the closed-form partial derivatives of $\Psi$ and the intermediate values at the three tangent points, so that each entry of Table~\ref{tab:tangent-bounds} can be checked without repeating the differentiation.

On the other hand, let $S_5(u)=u-u^3/6+u^5/120$; then $\sin u\le S_5(u)$ for $0\le u\le1$. Using \eqref{eq:p-Pm}, \eqref{eq:e-3m}, $\pi<\pi_+$, and the concavity of the sine,
with $p\le\Pi_{40}=101$ and $p\le\Pi_{41}=103$ in the two cases with $h=2$ and
$e\le3\cdot40-6=114$ in the case $h=1$, we obtain the lower bounds for $\Lambda$ recorded in Table~\ref{tab:lambda-lower}.
\begin{table}[ht]
\centering
\caption{Lower bounds for $\Lambda$ at the three boundary
cases.}\label{tab:lambda-lower}
\begin{tabular}{c@{\qquad}c}
\toprule
$(m,h)$ & lower bound for $\Lambda$\\
\midrule
$(40,2)$ &
$\displaystyle
\frac{860}{861}\,41
-2\pi_+S_5\!\left(\frac{101\pi_+}{861}\right)
>38.6889$\\[3mm]
$(40,1)$ &
$\displaystyle
\frac{860}{861}\,41
-\pi_+S_5\!\left(\frac{114\pi_+}{861}\right)
>39.6829$\\[3mm]
$(41,2)$ &
$\displaystyle
\frac{902}{903}\,42
-2\pi_+S_5\!\left(\frac{103\pi_+}{903}\right)
>39.7498$\\
\bottomrule
\end{tabular}
\end{table}
In each of the three cases, $\Lambda$ is bounded above by the tangent-plane value and below by the corresponding entry of Table~\ref{tab:lambda-lower}, and the lower bound exceeds the upper bound. This contradiction proves the proposition.
\end{proof}

We can now prove the main theorem, which states, equivalently, that no wheel with $n\ge41$ is geodesic Leech.

\begin{proof}[Proof of Theorem~\ref{thm:main}]
Let $m=n-1$. The number $h$ of large spokes is $0$, $1$, or $2$. Proposition~\ref{prop:uniform-ranges} excludes all $m\ge40$ with $h=0$, all $m\ge41$ with $h=1$, and all $m\ge42$ with $h=2$. Proposition~\ref{lem:boundary} excludes the remaining cases $(m,h)=(40,1),(40,2),(41,2)$. Hence no $m\ge40$ occurs, and $n=m+1\le40$.
\end{proof}

\section{\texorpdfstring{Explicit constructions for $W_7$ through $W_{13}$}{Explicit constructions for W7 through W13}}\label{sec:certificates}

This section proves the left inclusion of Corollary~\ref{cor:range}. In Table~\ref{tab:certificates}, each sequence is written in cyclic order along the rim. Here $A$ gives the spoke labels and $B$ the rim labels. The rows $W_7,\ldots,W_{13}$ are new; the rows $W_5$ and $W_6$ transcribe the two labelings of~\cite[Figure~3]{LakshmananManattu2025} into the same format, so that all the labelings underlying Corollary~\ref{cor:range} are exhibited in one place.

The new rows were found by a computer search organized by
Proposition~\ref{thm:frame-completion}, which splits the problem exactly into the
choice of an admissible spoke frame $A$ and the choice of a residual cyclic
completion of $A$. The two halves are of unequal difficulty. Completing a given
frame is the easier step: the completion is found by backtracking, and
$\min T(A)$ must be a rim label, after which each admissible $b_k$ is
constrained by the requirement that $b_{k-1}+b_k$ lie again in $T(A)$, so
the search tree is small. Choosing the frame is the harder step: it is a
dense Sidon-type packing in which the $m(m-3)/2$ pair sums must fit without
repetition into the $m(m+1)/2$ values not occupied by the labels; the arithmetic
constraints of Appendix~\ref{app:identities}, in particular the congruence
$\sum_{t\in T(A)}t\equiv0\pmod 3$ of~\eqref{eq:residual-sum}, eliminate candidate
frames before any completion is attempted. The cost of the frame search grows
rapidly with $m$. The reimplementation in the companion repository anneals over
the underlying \emph{set} of spoke labels, with a cost that vanishes exactly
when certain necessary conditions hold for the set to admit a cyclic order
making it an admissible frame, among them the congruence above, and then
attempts to assemble each set of cost zero into such an order. With its default
seed, on one core of a desktop processor (Intel Core i5-14400F), it finds
geodesic Leech labelings of $W_5,\ldots,W_9$ in under a second each, of
$W_{10}$ in $9$ seconds, of $W_{11}$ in $47$ seconds, and of $W_{12}$ in $4$
minutes; for $W_{13}$, eight seeds run in parallel produced a labeling after
$4.6$ minutes, about $37$ processor-minutes in all. For $W_{14}$ we ran a
compiled port of the same search, whose agreement with the reference
implementation was checked on identical inputs, on eight seeds for $18$ hours
each, $144$ processor-hours in all. Its $1.18\times10^{12}$ annealing steps
produced $6.0\times10^{6}$ sets of cost zero, and not one of them assembled
into an admissible spoke frame, so no rim completion was ever attempted: at
$m=13$ these necessary conditions are far from sufficient. The labelings found
for $W_5,\ldots,W_{13}$ differ from those of Table~\ref{tab:certificates},
which are far from unique. Deciding $W_{14}$ remains open; see Section~\ref{sec:conclusion}.

The companion repository \url{https://github.com/junyeobe0315/geodesic-leech-wheels}
provides the labelings in machine-readable form, together with code that
verifies them directly from the weighted graph and an independent
reimplementation of the search (not the program that produced
Table~\ref{tab:certificates}) for the cases left open in
Section~\ref{sec:conclusion}. No result of this paper depends on the search: the
existence claims rest on the explicit labelings of Table~\ref{tab:certificates},
whose geodesic weight classes are listed in Appendix~\ref{app:partitions}.

\begin{theorem}\label{thm:explicit}
For every $5\le n\le13$, the sequences of Table~\ref{tab:certificates} form a geodesic Leech labeling of $W_n$. In particular $\set{5,6,\ldots,13}\subseteq\calE$.
\end{theorem}

\noindent
The cases $n=5$ and $n=6$ are due to Lakshmanan S. and
Manattu~\cite{LakshmananManattu2025}; the cases $7\le n\le13$ are new.

\begin{proof}
For each row we compute the four classes of \eqref{eq:wheel-master}. It suffices to check that the result is exactly $[N]$, with no value repeated. Appendix~\ref{app:partitions} records the three geodesic weight classes for each $n$.

In the smallest new case $W_7$ we have $N=27$, and the weights of the one-edge geodesics, of the two-rim-edge geodesics, and of the hub-type two-edge geodesics are, respectively,
\begin{align*}
&\set{1,2,3,4,5,7,8,10,14,15,17,22},\\
&\set{11,13,16,20,23,25},\\
&\set{6,9,12,18,19,21,24,26,27}.
\end{align*}
These three sets are pairwise disjoint and their union is $[27]$. The other rows are verified in the same way from the weight classes in Appendix~\ref{app:partitions}.
\end{proof}

\begin{table}[ht]
\centering
\caption{Geodesic Leech labelings of $W_5,\ldots,W_{13}$. The rows $W_5$ and
$W_6$ are those of~\cite[Figure~3]{LakshmananManattu2025}; the remaining rows are
new.}\label{tab:certificates}
\begin{tabular}{@{}lll@{}}
\toprule
Wheel & Spoke labels $A$ & Rim labels $B$\\
\midrule
$W_5$ & $(6,2,7,3)$ & $(1,10,4,8)$\\
$W_6$ & $(14,7,6,5,11)$ & $(1,3,13,2,8)$\\
$W_7$ & $(4,5,2,14,22,7)$ & $(15,8,17,3,10,1)$\\
$W_8$ & $(9,6,7,21,13,8,5)$ & $(10,24,4,31,1,2,23)$\\
$W_9$ & $(5,8,1,13,34,28,9,3)$ & $(12,20,24,2,25,15,23,7)$\\
$W_{10}$ & $(15,4,19,39,29,1,7,13,2)$ & $(8,10,35,12,25,24,27,23,30)$\\
$W_{11}$ & $(1,32,33,16,22,9,27,4,29,31)$ & $(2,44,21,18,6,8,3,12,7,50)$\\
$W_{12}$ & $(10,31,17,22,26,13,38,39,2,37,3)$ & $(1,8,66,11,7,14,58,4,67,6,45)$\\
$W_{13}$ & $(25,10,36,16,30,44,2,42,3,43,46,21)$ & $(4,7,15,74,1,34,50,20,9,8,6,77)$\\
\bottomrule
\end{tabular}
\end{table}

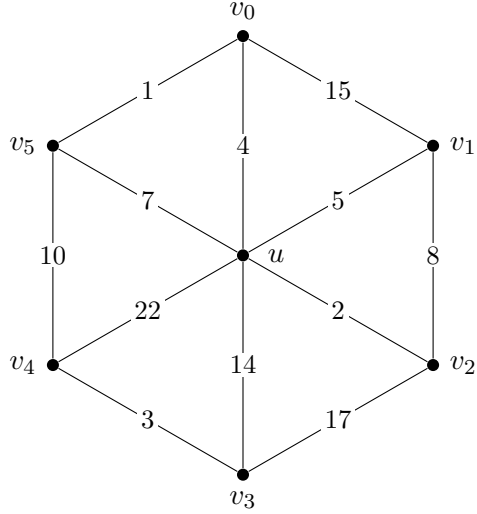
\begin{figure}[ht]
\centering
\begin{tikzpicture}[scale=1.45,
    vtx/.style={circle,fill,inner sep=1.6pt},
    el/.style={midway,fill=white,inner sep=1.3pt,font=\small}]
  \node[vtx,label={[label distance=3pt]0:{$u$}}] (u) at (0,0) {};
  \node[vtx,label={above:{$v_0$}}]     (v0) at (90:2)    {};
  \node[vtx,label={right:{$v_1$}}]     (v1) at (30:2)    {};
  \node[vtx,label={right:{$v_2$}}]     (v2) at (-30:2)   {};
  \node[vtx,label={below:{$v_3$}}]     (v3) at (-90:2)   {};
  \node[vtx,label={left:{$v_4$}}]      (v4) at (-150:2)  {};
  \node[vtx,label={left:{$v_5$}}]      (v5) at (150:2)   {};
  \foreach \i/\a in {0/4,1/5,2/2,3/14,4/22,5/7}
    \draw (u) -- (v\i) node[el] {$\a$};
  \foreach \i/\j/\b in {0/1/15,1/2/8,2/3/17,3/4/3,4/5/10,5/0/1}
    \draw (v\i) -- (v\j) node[el] {$\b$};
\end{tikzpicture}
\caption{The geodesic Leech labeling of $W_7$ in the first new row of
Table~\ref{tab:certificates}: the spoke $uv_i$ carries $a_i$ and the rim edge
$v_iv_{i+1}$ carries $b_i$, indices modulo $6$. The $27$ geodesic weights are
$1,2,\ldots,27$, each occurring once.}\label{fig:w7}
\end{figure}

Figure~\ref{fig:w7} shows the labeling of $W_7$, the case left open in~\cite[Figure~4]{LakshmananManattu2025}.

\begin{proof}[Proof of Corollary~\ref{cor:range}]
Theorem~\ref{thm:explicit} gives $\set{5,6,\ldots,13}\subseteq\calE$, and
Theorem~\ref{thm:main} gives $\calE\subseteq\set{5,6,\ldots,40}$.
\end{proof}

\section{Conclusion and open problems}\label{sec:conclusion}

The range obtained in this paper is
\[
  \set{5,6,\ldots,13}\subseteq\calE
  \subseteq\set{5,6,\ldots,40}.
\]
The left inclusion is proved by the explicit labelings of Theorem~\ref{thm:explicit}, and the right inclusion by the finite Fourier argument of Sections~\ref{sec:strategy}--\ref{sec:boundary}. Moreover, the frame-completion criterion (Proposition~\ref{thm:frame-completion}) decomposes the search over the remaining cases exactly into the choice of an admissible spoke frame and a cyclic completion by the $2m$ residual integers, and the arithmetic constraints of Appendix~\ref{app:identities} reduce that search space further.

\subsection*{What forces the bound \texorpdfstring{$40$}{40}}

The bound $40$ of Theorem~\ref{thm:main} is the threshold of a single
comparison, and we record which one, since it determines what a sharpening
would have to improve. In the case $h=0$ of Section~\ref{ssec:h0} the kernel sum $\Lambda$ is
bounded below by $\mathcal L_0(m)=m+1-2(m+1)/(m(m+3))$ and above by
$\mathcal U_0$ of \eqref{eq:U0}, and $\mathcal U_0(m)=c_{\mathcal U}\sqrt m+O(1)$ with
\[
  c_{\mathcal U}=\sqrt6\left(\frac{355}{226}+\frac23\right)+\frac{\sqrt7}3=6.3625\ldots
\]
(the same expression with $\pi$ in place of $355/113$ differs by less than
$10^{-6}$ and shares every digit displayed here). The two leading terms $m+1$
and $c_{\mathcal U}\sqrt m$ cross at
\[
  m=\left(\frac{c_{\mathcal U}+\sqrt{c_{\mathcal U}^2-4}}2\right)^{\!2}=38.456\ldots,
\]
and the exact margin $\mathcal L_0-\mathcal U_0$ is positive at $m=38$, though
only by $0.0230\ldots$. Proposition~\ref{prop:uniform-ranges}(i) is stated from
$m=40$ because that is where the hypotheses of
\eqref{eq:phase-pair-uniform}, and with them the upper bound \eqref{eq:U0}, are
verified: at $m=38$ one has $\Gamma=6m+2=230<242$.

Neither the interval-occupancy bound of Proposition~\ref{lem:rho-closed} nor the
localization of the large spokes in Section~\ref{sec:support-localization}
enters that comparison. What they control are the phase losses in the
cases of one and two large spokes, which amount to $1.27\ldots$ and
$2.26\ldots$ at $m=41$; against a margin growing by about $\tfrac12$ per unit of
$m$, such losses move the threshold, and the corresponding margins turn positive
only at $m=41$ and at $m=42$. The six-variable Parseval argument of
Section~\ref{sec:boundary} then removes the three cases $(m,h)=(40,1)$,
$(40,2)$, $(41,2)$ that are left over. Sharpening Proposition~\ref{lem:rho-closed} or
Section~\ref{sec:support-localization} alone would accordingly not lower the
bound of Theorem~\ref{thm:main}.

Sharpening $\mathcal L_0$ or $\mathcal U_0$ would improve
Proposition~\ref{prop:uniform-ranges}(i), but by itself it would not lower the
bound either, because the cases $(m,h)=(39,1),(39,2)$ would still have to be
excluded, and the argument of Section~\ref{sec:boundary} requires the modulus $M=(m+1)(m+2)/2$ to be odd, which happens exactly
for $m\equiv0,1\pmod4$. The three boundary cases lie at $m=40,41$, that is, at
$W_{41}$ and $W_{42}$, where $M=861$ and $M=903$; but $m=39$ gives $M=820$, for
which neither the summation to $(M-1)/2$ in \eqref{eq:parseval-q} nor the index
$\sigma_M$ is available as written. Lowering the bound by the present method
would therefore require both a sharper kernel comparison and a boundary argument
valid at even moduli. We have no evidence bearing on the true value of
$\max\calE$.

The cases $W_{14},\ldots,W_{40}$ remain undecided. The search of
Section~\ref{sec:certificates} found no labeling of $W_{14}$, and we have no
argument excluding one; since nonexistence at one order is not known to imply
nonexistence at larger orders, the remaining cases would have to be settled one
at a time, unless $\calE$ is an interval, which we cannot prove either. Whether
the bound $40$ can be lowered is discussed above; we have no evidence as to the
true value of $\max\calE$.

\section*{Data and code availability}
\phantomsection\addcontentsline{toc}{section}{Data and code availability}

The following are openly available in the companion repository at
\url{https://github.com/junyeobe0315/geodesic-leech-wheels}, whose tagged
releases are archived on Zenodo under the concept identifier
\href{https://doi.org/10.5281/zenodo.22254583}{doi:10.5281/zenodo.22254583};
the version corresponding to this article is \texttt{v1.0.4}.
\begin{itemize}[leftmargin=*]
\item The labelings of Table~\ref{tab:certificates} in machine-readable form.
\item Programs that verify them directly from the weighted graph, by enumerating
every shortest path of the weighted wheel and comparing the resulting weight
multiset against $[N]$, without appeal to
Proposition~\ref{prop:wheel-geodesics}.
\item A program that checks the weight classes of Appendix~\ref{app:partitions}
elementwise against the source of this article, so that a transcription error
would be detected.
\item A program that recomputes, in exact rational interval arithmetic with
outward rounding, every numerical comparison of
Sections~\ref{sec:fourier}--\ref{sec:boundary},
Appendix~\ref{app:kappa-iteration}, and Appendix~\ref{app:tangent}, including
the rational brackets for $\pi$, which it certifies from Machin's formula rather
than assuming.
\item An independent reimplementation of the frame search of
Section~\ref{sec:certificates} for the cases $W_{14},\ldots,W_{40}$, in
Python, together with a C++ port of the same search that was used for the
$W_{14}$ computation reported there, and a self-test mode in which the two
implementations can be compared on identical inputs.
\end{itemize}

\appendix
\section{Geodesic weight classes of the explicit labelings}\label{app:partitions}

For each row of Table~\ref{tab:certificates}, this appendix lists the values of \eqref{eq:wheel-master} grouped into three classes: $\mathcal S_1$ collects all one-edge geodesic weights, $\mathcal S_2$ the weights of two consecutive rim edges, and $\mathcal S_3$ the hub-type two-edge geodesic weights. In each case the three sets are pairwise disjoint and their union is $[N]$.

\subsection*{$W_5$ ($N=14$)}
\begin{align*}
\mathcal S_1={}&\{1,2,3,4,6,7,8,10\},\\
\mathcal S_2={}&\{9,11,12,14\},\\
\mathcal S_3={}&\{5,13\}.
\end{align*}

\subsection*{$W_6$ ($N=20$)}
\begin{align*}
\mathcal S_1={}&\{1,2,3,5,6,7,8,11,13,14\},\\
\mathcal S_2={}&\{4,9,10,15,16\},\\
\mathcal S_3={}&\{12,17,18,19,20\}.
\end{align*}

\subsection*{$W_7$ ($N=27$)}
\begin{align*}
\mathcal S_1={}&\{1,2,3,4,5,7,8,10,14,15,17,22\},\\
\mathcal S_2={}&\{11,13,16,20,23,25\},\\
\mathcal S_3={}&\{6,9,12,18,19,21,24,26,27\}.
\end{align*}

\subsection*{$W_8$ ($N=35$)}
\begin{align*}
\mathcal S_1={}&\{1,2,4,5,6,7,8,9,10,13,21,23,24,31\},\\
\mathcal S_2={}&\{3,25,28,32,33,34,35\},\\
\mathcal S_3={}&\{11,12,14,15,16,17,18,19,20,22,26,27,29,30\}.
\end{align*}

\subsection*{$W_9$ ($N=44$)}
\begin{align*}
\mathcal S_1={}&\{1,2,3,5,7,8,9,12,13,15,20,23,24,25,28,34\},\\
\mathcal S_2={}&\{19,26,27,30,32,38,40,44\},\\
\mathcal S_3={}&\{4,6,10,11,14,16,17,18,21,22,29,31,33,35,36,37,\\
&\qquad 39,41,42,43\}.
\end{align*}

\subsection*{$W_{10}$ ($N=54$)}
\begin{align*}
\mathcal S_1={}&\{1,2,4,7,8,10,12,13,15,19,23,24,25,27,29,30,35,39\},\\
\mathcal S_2={}&\{18,37,38,45,47,49,50,51,53\},\\
\mathcal S_3={}&\{3,5,6,9,11,14,16,17,20,21,22,26,28,31,32,33,34,36,\\
&\qquad40,41,42,43,44,46,48,52,54\}.
\end{align*}

\subsection*{$W_{11}$ ($N=65$)}
\begin{align*}
\mathcal S_1={}&\{1,2,3,4,6,7,8,9,12,16,18,21,22,27,29,31,32,33,44,50\},\\
\mathcal S_2={}&\{11,14,15,19,24,39,46,52,57,65\},\\
\mathcal S_3={}&\{5,10,13,17,20,23,25,26,28,30,34,35,36,37,38,40,41,42,\\
&\qquad43,45,47,48,49,51,53,54,55,56,58,59,60,61,62,63,64\}.
\end{align*}

\subsection*{$W_{12}$ ($N=77$)}
\begin{align*}
\mathcal S_1={}&\{1,2,3,4,6,7,8,10,11,13,14,17,22,26,31,37,38,39,45,58,66,67\},\\
\mathcal S_2={}&\{9,18,21,46,51,62,71,72,73,74,77\},\\
\mathcal S_3={}&\{5,12,15,16,19,20,23,24,25,27,28,29,30,32,33,34,35,36,40,41,\\
&\qquad42,43,44,47,48,49,50,52,53,54,55,56,57,59,60,61,63,64,65,68,\\
&\qquad69,70,75,76\}.
\end{align*}

\subsection*{$W_{13}$ ($N=90$)}
\begin{align*}
\mathcal S_1={}&\{1,2,3,4,6,7,8,9,10,15,16,20,21,25,30,34,36,42,43,44,46,50,74,77\},\\
\mathcal S_2={}&\{11,14,17,22,29,35,70,75,81,83,84,89\},\\
\mathcal S_3={}&\{5,12,13,18,19,23,24,26,27,28,31,32,33,37,38,39,40,41,45,47,\\
&\qquad48,49,51,52,53,54,55,56,57,58,59,60,61,62,63,64,65,66,67,68,\\
&\qquad69,71,72,73,76,78,79,80,82,85,86,87,88,90\}.
\end{align*}

\section{Arithmetic constraints for the remaining cases}\label{app:identities}

This appendix collects necessary conditions satisfied by any geodesic Leech labeling. They are independent of the upper-bound proof, and they are the constraints that underlie the search of Section~\ref{sec:certificates}: since $\sum_{t\in T(A)}t$ and $\sum_{x\in\Phi(A)}x$ depend only on the underlying set of spoke labels, the congruence $\sum_{t\in T(A)}t\equiv0\pmod3$ of~\eqref{eq:residual-sum} eliminates candidate frames before any completion is attempted, while the parity relation~\eqref{eq:parity}, once the frame is fixed, prescribes the value of $\nu_B+\chi_B$ for every completion. They are recorded here because they are the principal reduction of the search space available for the cases $W_{14},\ldots,W_{40}$ left open in Section~\ref{sec:conclusion}. For the spoke labels $A=(a_i)$, the rim labels $B=(b_i)$, the quantity $N=m(m+3)/2$, and the residual set $T(A)$ of Section~\ref{sec:frame}, the notation is as in the main text; the polynomials and quantities defined below are local to this appendix.

\subsection{Generating function identities}

Define the following polynomials.
\[
  P_A(z)=\sum_i z^{a_i},\qquad P_B(z)=\sum_i z^{b_i},
\]
\[
  C_A(z)=\sum_i z^{a_i+a_{i+1}},\qquad
  C_B(z)=\sum_i z^{b_i+b_{i+1}}.
\]

\begin{proposition}[Basic generating-function identity]\label{prop:master-generating}
A geodesic Leech labeling satisfies the following identity.
\begin{equation}\label{eq:master-generating}
  \sum_{r=1}^{N}z^r
  =P_A(z)+P_B(z)+C_B(z)
  +\frac{P_A(z)^2-P_A(z^2)}2-C_A(z).
\end{equation}
\end{proposition}

\begin{proof}
The polynomials $P_A(z)$ and $P_B(z)$ represent the one-edge spoke and rim geodesics, respectively, and $C_B(z)$ represents the sums of two consecutive rim edges. On the other hand,
\[
  \frac{P_A(z)^2-P_A(z^2)}2
\]
counts the unordered sums of all distinct spoke pairs; among these, the pairs adjacent on the rim are not hub-type geodesics, so $C_A(z)$ is subtracted. By \eqref{eq:wheel-master}, the remaining exponents form exactly $1,\ldots,N$, each occurring once.
\end{proof}

Applying differentiation, sums of squares, or substitution of roots of unity to \eqref{eq:master-generating} recovers several of the conditions below.

\subsection{Sums and sums of squares}

\begin{proposition}\label{prop:moment-identities}
Let $A=(a_i)$ and $B=(b_i)$ be the spoke and rim labels of a geodesic Leech labeling. Then
\begin{equation}\label{eq:first-moment}
  (m-2)\sum_i a_i+3\sum_i b_i=\frac{N(N+1)}2,
\end{equation}
\begin{equation}\label{eq:residual-sum}
  \sum_{t\in T(A)}t=3\sum_i b_i,
\end{equation}
and
\begin{equation}\label{eq:second-moment}
\begin{aligned}
  \frac{N(N+1)(2N+1)}6
  ={}&\left(\sum_i a_i\right)^2
  +(m-3)\sum_i a_i^2
  -2\sum_i a_i a_{i+1}\\
  &+3\sum_i b_i^2
  +2\sum_i b_i b_{i+1}
\end{aligned}
\end{equation}
hold. In particular,
\[
  \sum_{t\in T(A)}t\equiv0\pmod3.
\]
\end{proposition}

\begin{proof}
The sum of all geodesic weights is $N(N+1)/2$. A fixed spoke label $a_i$ appears once by itself and $m-3$ times in sums with nonadjacent spokes, for a total coefficient of $m-2$. A fixed rim label $b_i$ appears once by itself and twice in the two consecutive rim sums, for a total coefficient of $3$. This gives \eqref{eq:first-moment}.

From Proposition~\ref{thm:frame-completion},
\[
  \sum_{t\in T(A)}t
  =\sum_i b_i+\sum_i(b_i+b_{i+1})
  =3\sum_i b_i,
\]
so \eqref{eq:residual-sum} and the congruence follow.

The left-hand side of the square-sum identity is $\sum_{r=1}^N r^2$. Expand the spoke part
\[
  \sum_i a_i^2+
  \sum_{\{i,j\}\notin E(C_m)}(a_i+a_j)^2
\]
and the rim part
\[
  \sum_i b_i^2+
  \sum_i(b_i+b_{i+1})^2.
\]
For the spoke part,
\[
  \sum_{i<j}(a_i+a_j)^2=(m-2)\sum_ia_i^2+\Bigl(\sum_ia_i\Bigr)^2,
  \qquad
  \sum_i(a_i+a_{i+1})^2=2\sum_ia_i^2+2\sum_ia_ia_{i+1},
\]
and the sum over the pairs $\{i,j\}\notin E(C_m)$ is the difference of the two;
adding the singleton term $\sum_ia_i^2$ turns the coefficient $m-4$ of
$\sum_ia_i^2$ into $m-3$, so that
\[
  \sum_ia_i^2+\sum_{\{i,j\}\notin E(C_m)}(a_i+a_j)^2
  =\Bigl(\sum_ia_i\Bigr)^2+(m-3)\sum_ia_i^2-2\sum_ia_ia_{i+1},
\]
which is the first line of \eqref{eq:second-moment}. The rim expansion
$\sum_ib_i^2+\sum_i(b_i+b_{i+1})^2=3\sum_ib_i^2+2\sum_ib_ib_{i+1}$ yields the
second.
\end{proof}

\subsection{Parity conditions}

Let $\nu_A$ be the number of odd spoke labels and $\nu_B$ the number of odd rim labels. Let $\chi_A$ and $\chi_B$ denote the numbers of adjacent pairs at which the parity changes in the two cyclic sequences, respectively.

\begin{proposition}\label{prop:parity}
A geodesic Leech labeling satisfies
\begin{equation}\label{eq:parity}
  \nu_A+\nu_B+\nu_A(m-\nu_A)-\chi_A+\chi_B=\ceil{N/2}.
\end{equation}
Moreover, $\chi_A$ and $\chi_B$ are even.
\end{proposition}

\begin{proof}
Among the one-edge geodesics there are $\nu_A+\nu_B$ odd weights. Among all distinct spoke pairs, $\nu_A(m-\nu_A)$ produce an odd sum, but $\chi_A$ of them are adjacent spoke pairs of different parity and hence not hub-type geodesics, so they are subtracted. Among the consecutive rim sums, $\chi_B$ are odd. Since $[N]$ contains $\ceil{N/2}$ odd numbers, \eqref{eq:parity} holds. A cyclic binary sequence must return to its initial value, so the number of value changes is even.
\end{proof}

\section{Fixed-point iteration at the boundary values}\label{app:kappa-iteration}

This appendix exhibits the iteration that produces
Table~\ref{tab:kappa-maxima}, so that its entries can be read off at the
critical values of $d$ without repeating the full computation. Throughout,
$m\in\set{40,41,42}$ and the notation is that of
Section~\ref{sec:support-localization}.

Fix $m$ and fix the difference $d\ge1$. Put $\kappa_0=5m-7$, the a priori bound
established before Lemma~\ref{lem:two-large-position}, and define
\begin{equation}\label{eq:kappa-iterate}
  \kappa_{j+1}=\floor*{\frac{3m}{2}}+\mu(d)+2\mu(\kappa_j).
\end{equation}
Since $\kappa\le\kappa_0$, inequality \eqref{eq:z-fixedpoint} together with the
monotonicity \eqref{eq:rho-monotone} of $\mu$ gives $\kappa\le\kappa_j$ for every
$j$ by induction, and only these inequalities are used. In every case computed
the sequence is nonincreasing and becomes constant after at most four steps; the
first value that repeats is the quantity denoted $\kappa(d)$ in
Lemma~\ref{lem:two-large-position}. Here $\mu(\kappa_j)$
is evaluated from \eqref{eq:rho-sharp}, that is, $\mu(w)$ is the largest integer
$\ell$ with $(1+\lambda_w^2)\ell^2-(5+\lambda_w)\ell+6-4w\le0$ at $w=\kappa_j$, every
argument occurring below being at least $103$; the value $\mu(d)$ is taken from
\eqref{eq:rho-small} for $d\le9$ and is the smaller of \eqref{eq:rho-elementary}
and \eqref{eq:rho-sharp} for $d\ge10$, the two agreeing at the values of $d$
displayed. Finally, $d$ is retained only while the accompanying constraint
$d<\kappa/2$ of \eqref{eq:z-fixedpoint} is still satisfiable, that is, only while
$2d<\kappa(d)$.

\begin{table}[ht]
\centering
\caption{The iteration \eqref{eq:kappa-iterate} at the critical values
of $d$. The last two columns are the quantities maximized in
Table~\ref{tab:kappa-maxima}. A dagger marks a value of $d$ that is discarded
because $2d\ge\kappa(d)$.}\label{tab:kappa-iterates}
\begin{tabular}{ccccccc}
\toprule
$m$ & $d$ & $\mu(d)$ & $\kappa_0,\kappa_1,\ldots,\kappa(d)$
& $\mu(\kappa(d))$ & $\kappa(d)$ & $\kappa(d)-d-1$\\
\midrule
$40$ & $1$ & $1$ & $193,\,117,\,105,\,103$ & $21$ & $103$ & $101$\\
$40$ & $5$ & $5$ & $193,\,121,\,109,\,107$ & $21$ & $107$ & $101$\\
$40$ & $54$ & $16$ & $193,\,132,\,122,\,120$ & $22$ & $120$ & $65$\\
$40$ & $62^{\dagger}$ & $17$ & $193,\,133,\,123$ & $23$ & $123$ & $60$\\
\midrule
$41$ & $1$ & $1$ & $198,\,118,\,106,\,104$ & $21$ & $104$ & $102$\\
$41$ & $5$ & $5$ & $198,\,122,\,110,\,108$ & $21$ & $108$ & $102$\\
$41$ & $54$ & $16$ & $198,\,133,\,123$ & $23$ & $123$ & $68$\\
$41$ & $62^{\dagger}$ & $17$ & $198,\,134,\,124$ & $23$ & $124$ & $61$\\
\midrule
$42$ & $1$ & $1$ & $203,\,120,\,108,\,106$ & $21$ & $106$ & $104$\\
$42$ & $5$ & $5$ & $203,\,124,\,114,\,112$ & $22$ & $112$ & $106$\\
$42$ & $54$ & $16$ & $203,\,135,\,127,\,125$ & $23$ & $125$ & $70$\\
$42$ & $62$ & $17$ & $203,\,136,\,128,\,126$ & $23$ & $126$ & $63$\\
\bottomrule
\end{tabular}
\end{table}

The final entry $\kappa(d)$ of each row of Table~\ref{tab:kappa-iterates} is a
fixed point of \eqref{eq:kappa-iterate}: for instance the third row
reads $60+16+2\cdot22=120$, and the last reads $63+17+2\cdot23=126$. The two rows
marked with a dagger are discarded, since there $2d\ge\kappa(d)$, namely
$124>123$ at $m=40$ and $124=124$ at $m=41$; this excludes the large values of
$d$. The same $d=62$ survives at $m=42$, where $124<126$, and it is there that
the first maximum is attained.

Carrying the iteration over all $d$ leaves the admissible ranges $1\le d\le59$
at $m=40$, $1\le d\le61$ at $m=41$, and $1\le d\le62$ at $m=42$. On these
ranges, $\max_d\kappa(d)$ equals $120$, attained exactly for $54\le d\le59$;
$123$, attained exactly for $54\le d\le61$; and $126$, attained only at $d=62$.
Likewise $\max_d\bigl(\kappa(d)-d-1\bigr)$ equals $101$, attained exactly for
$1\le d\le5$; $102$, attained exactly for $d\in\set{1,2,3,4,5,9}$; and $106$,
attained only at $d=5$. These six maxima are the entries of
Table~\ref{tab:kappa-maxima}, and the four values of $d$ displayed above were
chosen so that each of them is attained in a displayed row. The entries of both
tables were computed by a program, in integer arithmetic throughout, for every
admissible $d$ and not only for the twelve rows above.

At $m=40$ the second maximum equals $\Pi_{40}=101$ and at $m=42$ it equals
$\Pi_{42}=106$, so \eqref{eq:p-Pm} is attained with equality at two of the three
boundary values; only at $m=41$ is there slack, the maximum $102$ being one less
than $\Pi_{41}=103$. The equality at $m=40$ is used without slack. In the case
$(m,h)=(40,2)$ of Proposition~\ref{lem:boundary} the
lower-bound expression of Table~\ref{tab:lambda-lower} equals
$38.68891\ldots$, exceeding the upper bound $38.682$ of
Table~\ref{tab:tangent-bounds} by less than $0.007$, whereas replacing $101$ by
$102$ in that expression lowers it to $38.66754\ldots$, below $38.682$. The
comparison of Proposition~\ref{lem:boundary} would therefore fail at $(40,2)$ if
\eqref{eq:p-Pm} were weakened by one at $m=40$.

\section{Supporting values for the tangent planes}\label{app:tangent}

This appendix records the data behind the tangent-plane table of
Section~\ref{sec:boundary}: the closed-form partial derivatives of $\Psi$
and the intermediate values at the three tangent points. With these, every
entry of the table reduces to finitely many arithmetic operations and square
roots, each certifiable by squaring both sides as in the main text. The
intermediate values below are displayed rounded to five decimal places, hence
to within $10^{-5}$; they are given to identify the quantities entering the
derivative formulas, not as the inputs of a calculation. The gradient entries
that follow are obtained by evaluating those formulas at the exact rational
tangent point, so the small components $g_4,g_5,g_6$ are not limited by the
display precision of Table~\ref{tab:app-values}. Nothing in the proof of
Proposition~\ref{lem:boundary} rests on the displayed digits: the statements on which
the proof rests are the rational inequalities tabulated in
Section~\ref{sec:boundary}, each verified by squaring both sides. As in
Appendix~\ref{app:identities}, the abbreviations introduced here are local
to this appendix.

\subsection*{Partial derivatives of \texorpdfstring{$\Psi$}{Psi}}

Fix $(m,h)$ and abbreviate, at a point $x=(q_1,q_2,q_4,q_6,q_8,q_{12})$,
\[
  w_4=\sqrt{2q_4+2m+2+U_m^*},
  \qquad
  v_2=v_2(x),
  \qquad
  \Theta=2q_1+2\Db_1-v_2,
\]
\[
  \cV^{(1)}=\cV_m(q_4,q_8),\quad
  \cV^{(2)}=\cV_m(q_6,q_{12}),\quad
  \cV^{(3)}=\cV_m(\cX_m(x),\cY_m(x)),
\]
\[
  w_8=\sqrt{2q_8+2m+2+U_m^*},\qquad
  w_{12}=\sqrt{2q_{12}+2m+2+U_m^*},
\]
\[
  w_{\cY}=\sqrt{2\cY_m(x)+2m+2+U_m^*},
\]
and write $\eta=\eta(x)$. From \eqref{eq:X-tail}--\eqref{eq:Y-tail},
\[
  \frac{\partial \cX_m}{\partial q_j}
  =\frac{\hat c_j^{(2)}-\frac{267}{10000}\,q_j/\eta}{\omega_m},
  \qquad
  \frac{\partial \cY_m}{\partial q_j}
  =\frac{\hat c_j^{(4)}-\frac{219}{5000}\,q_j/\eta}{\omega_m},
\]
where $\hat c_j^{(2)}$ is the coefficient of $q_j$ in the numerator of $\cX_m$ (namely
$c_R$, $c_4$, $c_6$ for $q_1$, $q_8$, $q_{12}$, and $0$ otherwise) and
$\hat c_j^{(4)}$ is the analogous coefficient of $\cY_m$ (namely $c_{\sigma_M}$, $c_R$,
$c_{R-1}$ for $q_1$, $q_2$, $q_6$, and $0$ otherwise). Differentiating \eqref{eq:Psi} by the
chain rule then gives
\begin{align*}
  \frac{\partial\Psi}{\partial q_1}
  &=\frac{\pi_+}{2\sqrt\Theta}
    +\omega_m\left[
      \frac1{\cV^{(3)}}\frac{\partial \cX_m}{\partial q_1}
      +\frac1{2\cV^{(3)}w_{\cY}}\frac{\partial \cY_m}{\partial q_1}
    \right],\\
  \frac{\partial\Psi}{\partial q_2}
  &=\left(\frac23-\frac{\pi_+}{4\sqrt\Theta}\right)\frac1{v_2}
    +\omega_m\left[
      \frac1{\cV^{(3)}}\frac{\partial \cX_m}{\partial q_2}
      +\frac1{2\cV^{(3)}w_{\cY}}\frac{\partial \cY_m}{\partial q_2}
    \right],\\
  \frac{\partial\Psi}{\partial q_4}
  &=\left(\frac23-\frac{\pi_+}{4\sqrt\Theta}\right)\frac1{2v_2w_4}
    +\frac{c_2}{\cV^{(1)}}
    +\omega_m\left[
      \frac1{\cV^{(3)}}\frac{\partial \cX_m}{\partial q_4}
      +\frac1{2\cV^{(3)}w_{\cY}}\frac{\partial \cY_m}{\partial q_4}
    \right],\\
  \frac{\partial\Psi}{\partial q_6}
  &=\frac{c_3}{\cV^{(2)}}
    +\omega_m\left[
      \frac1{\cV^{(3)}}\frac{\partial \cX_m}{\partial q_6}
      +\frac1{2\cV^{(3)}w_{\cY}}\frac{\partial \cY_m}{\partial q_6}
    \right],\\
  \frac{\partial\Psi}{\partial q_8}
  &=\frac{c_2}{2\cV^{(1)}w_8}
    +\omega_m\left[
      \frac1{\cV^{(3)}}\frac{\partial \cX_m}{\partial q_8}
      +\frac1{2\cV^{(3)}w_{\cY}}\frac{\partial \cY_m}{\partial q_8}
    \right],\\
  \frac{\partial\Psi}{\partial q_{12}}
  &=\frac{c_3}{2\cV^{(2)}w_{12}}
    +\omega_m\left[
      \frac1{\cV^{(3)}}\frac{\partial \cX_m}{\partial q_{12}}
      +\frac1{2\cV^{(3)}w_{\cY}}\frac{\partial \cY_m}{\partial q_{12}}
    \right].
\end{align*}

\subsection*{Constants}

For $m=40$: $M=861$, $R=430$, $\sigma_M=215$, $E_{40}=31240$,
$U_{40}^*=\frac{1+\sqrt{969}}2=16.06438\ldots$,
$\omega_{40}=\frac{122}{861}=0.141695\ldots$, and
\[
  c_{430}=\frac{2}{739599},\qquad
  c_{429}=\frac{2}{736163},\qquad
  c_{215}=\frac{2}{184899}.
\]
For $m=41$: $M=903$, $R=451$, $\sigma_M=226$, $E_{41}=33661$,
$U_{41}^*=\frac{1+\sqrt{993}}2=16.25595\ldots$,
$\omega_{41}=\frac{128}{903}=0.141749\ldots$, and
\[
  c_{451}=\frac{2}{813603},\qquad
  c_{450}=\frac{2}{809999},\qquad
  c_{226}=\frac{2}{204303}.
\]
The pairs $(\Db_1,\Db_2)$ are $(\frac{157}4,\frac{81}2)$ for $(40,2)$,
$(41,41)$ for $(40,1)$, and $(\frac{161}4,\frac{83}2)$ for $(41,2)$.

\subsection*{Intermediate values at the tangent points}

\begin{table}[ht]
\centering
\caption{Intermediate values at the three tangent points, rounded to five
decimal places.}\label{tab:app-values}
\begin{tabular}{lccc}
\toprule
$(m,h)$ & $(40,2)$ & $(40,1)$ & $(41,2)$\\
\midrule
$\eta$        & $49.44534$ & $49.44468$ & $52.25463$\\
$\cX_m$         & $25.58936$ & $25.59767$ & $27.04591$\\
$\cY_m$         & $15.29318$ & $15.29297$ & $16.15482$\\
$w_4$         & $15.99326$ & $15.99326$ & $16.18752$\\
$v_2$         & $15.95974$ & $15.99104$ & $16.15449$\\
$\sqrt\Theta$ & $14.84117$ & $14.95757$ & $15.03747$\\
$\cV^{(1)}$         & $15.95401$ & $15.95405$ & $16.15228$\\
$\cV^{(2)}$         & $15.67251$ & $15.67070$ & $16.00200$\\
$w_{\cY}$         & $11.34243$ & $11.34241$ & $11.51371$\\
$\cV^{(3)}$           & $12.02170$ & $12.02239$ & $12.23133$\\
$\Psi(x_0)$   & $38.67848$ & $38.88219$ & $39.19229$\\
\bottomrule
\end{tabular}
\end{table}

Substituting these values into the derivative formulas above yields the gradients
recorded in Table~\ref{tab:app-gradients}.
\begin{table}[ht]
\centering
\caption{The gradient $g=\nabla\Psi$ at the three tangent
points.}\label{tab:app-gradients}
\begin{tabular}{lccc}
\toprule
$(m,h)$ & $(40,2)$ & $(40,1)$ & $(41,2)$\\
\midrule
$g_1$ & $0.102042$ & $0.101219$ & $0.100839$\\
$g_2$ & $0.034658$ & $0.034608$ & $0.034415$\\
$g_3$ & $0.005761$ & $0.005760$ & $0.005810$\\
$g_4$ & $2.77\cdot10^{-6}$ & $4.80\cdot10^{-6}$ & $2.13\cdot10^{-6}$\\
$g_5$ & $2.08\cdot10^{-6}$ & $1.06\cdot10^{-6}$ & $-0.90\cdot10^{-6}$\\
$g_6$ & $0.92\cdot10^{-6}$ & $-1.04\cdot10^{-6}$ & $-3.25\cdot10^{-6}$\\
\bottomrule
\end{tabular}
\end{table}
From these one obtains, for the three cases in the order $(40,2)$, $(40,1)$,
$(41,2)$,
\[
  A_0=\Psi(x_0)-g\cdot x_0<27.44363,\ 27.71625,\ 27.78160,
\]
\[
  g_1+g_2+g_3<0.142462,\ 0.141587,\ 0.141065,
\]
\[
  \|(g_4,g_5,g_6)\|_2<3.59\cdot10^{-6},\ 5.03\cdot10^{-6},\ 3.99\cdot10^{-6},
\]
in agreement with the rational bounds of Table~\ref{tab:tangent-bounds}. Each of
these nine decimals is an exact rational upper bound, obtained by rounding the
computed value up in the last digit.

\end{document}